\documentclass[11pt]{article}

\usepackage[T1]{fontenc}
\usepackage[utf8]{inputenc}
\usepackage{lmodern}
\usepackage{fullpage}
\usepackage{mathtools}
\usepackage{amsthm}
\usepackage{amssymb}
\usepackage{tikz-cd}
\usepackage{enumitem}

\usepackage{csquotes}

\usepackage{xcolor}
\usepackage{hyperref} 
\hypersetup{%
	colorlinks,%
	linkcolor={red!50!black},%
	citecolor={blue!50!black},%
	urlcolor={blue!80!black}%
}

\theoremstyle{plain}
\newtheorem{introthm}{Theorem}

\newtheorem{theorem}{Theorem}[subsection]
\newtheorem{lemma}[theorem]{Lemma}
\newtheorem{proposition}[theorem]{Proposition}
\newtheorem{corollary}[theorem]{Corollary}

\theoremstyle{definition}
\newtheorem{definition}[theorem]{Definition}

\newtheorem{construction}[theorem]{Construction}

\newtheorem{remark}[theorem]{Remark}

\numberwithin{equation}{section}

\usepackage{chngcntr}
\usepackage{euscript}
\usepackage[bbgreekl]{mathbbol}

\newcommand*\mathblank{\mathord{-}}
\newcommand{\op}{\mathrm{op}} 
\def\on{\operatorname}
\def\th{\on{t}}
\def\Cat{\on{Cat}}
\def\Fun{\on{Fun}}
\def\St{\on{St}}
\def\ECat{\EuScript{C}\!\on{at}}

\def\FCat{\mathfrak{C}\!\on{at}}
\def\Cf{\mathfrak{C}}
\def\Cfsc{\mathfrak{C}^{\on{sc}}}
\def\CC{\mathbb{C}}
\def\DD{\mathbb{D}}
\def\CCop{{\mathbb{C}^{(\op,\op)}}}
\def\ZZ{\mathbb{Z}}

\def\OO{\mathbb{O}}

\def\A{\EuScript{A}}
\def\L{\EuScript{L}}
\def\K{\EuScript{K}}
\def\U{\EuScript{U}}
\def\V{\EuScript{V}}

\def\D{\EuScript{D}}
\def\Dn{\EuScript{D}^n}
\def\Sn{\EuScript{S}^n}
\def\Lni{\EuScript{L}^n_i}
\def\Set{\on{Set}}
\def\sSet{{\on{Set}_{\Delta}}}
\def\msSet{{\on{Set}_{\Delta}^+}}

\def\B{\EuScript{B}}
\def\C{\EuScript{C}}
\def\N{\on{N}}
\def\P{\EuScript{P}}
\def\lra{\longrightarrow}
\def\lla{\longleftarrow}
\def\hra{\hookrightarrow}
\def\llra{\def\arraystretch{.1}\begin{array}{c} \lra \\ \lla \end{array}}
\def\Un{\operatorname{Un}}
\def\St{\operatorname{St}}
\def\Unsc{\operatorname{Un}^{\on{sc}}}
\def\Stsc{\operatorname{St}^{\on{sc}}}
\def\CUnsc{\operatorname{CUn}^{\on{sc}}}
\def\CStsc{\operatorname{CSt}^{\on{sc}}}

\def\cchi{\mathbb{\bbchi}}
\def\PPhi{\mathbb{\bbphi}}
\DeclareMathSymbol\DDelta\mathord{bbold}{"01}
\DeclareMathSymbol\GGamma\mathord{bbold}{"00}
\DeclareMathSymbol\SSigma\mathord{bbold}{'117}

\setenumerate{label=(\arabic*)}

\DeclareMathOperator\sk{sk}

\DeclareMathOperator\Hom{Hom}

\DeclareMathOperator\id{id}

\DeclareMathOperator\Nsc{N^{sc}}
\DeclareMathOperator\colim{colim}	
\DeclareMathOperator\hocolim{hocolim}

\DeclareMathOperator\Map{Map}

\usepackage{glossaries}

\newglossaryentry{chi1}{
	name={$\chi_{C}$},
	description={The relative $1$-nerve}
}

\newglossaryentry{chi2}{
	name={$\cchi_{\CC}$},
	description={The relative $2$-nerve}
}

\newglossaryentry{phi2}{
	name={$\PPhi_{\CC}$},
	description={The left adjoint to the relative $2$-nerve functor}
}

\newglossaryentry{OI}{
	name={$\OO^I$},
	description={The free $2$-category on the $I$-simplex}
}

\newglossaryentry{DI}{
	name={$D^I$, $\D^I$},
	description={Poset used in defining the $(\infty,2)$ relative nerve}
}

\newglossaryentry{AJ}{
	name={$A(J)$, $\A(J)$},
	description={Subposets of $D^I$ corresponding to faces of simplices}
}

\newglossaryentry{Lni}{
	name={$\Lni$},
	description={Simplicial subset of $\D^n$ obtained from the horn $\Lambda^n_i$ under $\PPhi_{\OO^n}$}
}

\newglossaryentry{Ladj}{
	name={$\mathfrak{L}^n_i$},
	description={The image of $(\Lambda^n_i)^\flat\to \Delta^n$ under $\PPhi_{\OO^n}$}
}

\newglossaryentry{Dadj}{
	name={$\mathfrak{D}^n$},
	description={The image of $(\Delta^n)^\flat\to \Delta^n$ under $\PPhi_{\OO^n}$}}

\newglossaryentry{Stsc}{
	name={$\St_{\CC}$},
	description={Scaled straightening functor}
}
\newglossaryentry{Unsc}{
	name={$\Un_{\CC}$},
	description={Scaled unstraightening functor}
}

\newglossaryentry{Pn}{
	name={$P^n(S,T)$, $\P^n(S,T)$},
	description={Poset computing the homotopy type of mapping spaces in $\D^n$}
}

\newglossaryentry{x}{
	name={$x$},
	description={Embedding of $\D^n$ into the lattice $\ZZ^{n-1}$}
}

\newglossaryentry{etaC}{
	name={$\eta_{\CC}$},
	description={Comparison map between scaled straightening and the left adjoint to the relative nerve}
}

\newglossaryentry{BJ}{
	name={$\B(J)$},
	description={Thickened version of $\A(J)$}
}

\newglossaryentry{pI}{
	name={$p_I$, $\pi_I$},
	description={Functors comparing $\D^I$ to $\Delta^I$}
}

\newglossaryentry{rho}{
	name={$\rho_{I,J}$},
	description={Pullback functors, which relate $\D^I$ to $\D^J$}}

\renewcommand{\glossarysection}[2][]{}

\makeglossaries

\title{A relative $2$-nerve}
\author{Fernando Abell\'an Garc\'ia, Tobias Dyckerhoff, Walker H. Stern}
\begin{document}
\maketitle

\begin{abstract}
	In this work, we introduce a $2$-categorical variant of Lurie's relative nerve functor.  We
	prove that it defines a right Quillen equivalence which, upon passage to
	$\infty$-categorical localizations, corresponds to Lurie's scaled unstraightening
	equivalence. In this $\infty$-bicategorical context, the relative $2$-nerve provides a
	computationally tractable model for the Grothendieck construction which becomes equivalent,
	via an explicit comparison map, to Lurie's relative nerve when restricted to $1$-categories. 
\end{abstract}

\tableofcontents

\newpage
\section*{Introduction}
\addcontentsline{toc}{section}{Introduction}

Let $F: C^{\op} \to \Cat$ be a contravariant functor from a category $C$ into the category $\Cat$ of
small categories. The {\em classical Grothendieck construction} of $F$ is the category $\chi(F)$ with objects
given by pairs $(c,x)$, where $c \in C$ and $x \in F(c)$, and a morphism between objects $(c,x)$ and
$(c^\prime,x^\prime)$ consisting of a pair $(f,\eta)$, where $f:c \to c^\prime$ is a morphism in $C$ and $\eta: x \to
F(f)(x^\prime)$ is a morphism in the category $F(c)$.
This data admits an apparent composition law and the resulting category comes equipped with a forgetful
functor $\chi(F) \to C$ which is a Cartesian fibration. 
It is a classical result that the Grothendieck construction $\chi$ establishes an equivalence between
appropriately defined categories of 
\begin{enumerate}[label=(\Roman{*})]
	\item pseudo-functors $C^{\op} \to \Cat$,
	\item Cartesian fibrations over $C$.
\end{enumerate}
The construction originally appeared in the foundations of Grothendieck's theory of descent
\cite[Expos\'{e} VI]{Grothendieck:SGA1} where descent data is described in terms of sections of
Cartesian fibrations.\\

\noindent
{\bf The relative nerve.} 
Let $\sSet$ denote the category of simplicial sets and let $\Cat_{\infty} \subset
\sSet$ denote the full subcategory spanned by the $\infty$-categories. 
Generalizing the above, we may define the Grothendieck construction
of a contravariant functor 
\[
F: C^{\op} \lra \Cat_{\infty}
\]
from $C$ into the category $\Cat_{\infty}$
to be the simplicial set $\chi(F)$ in which an $n$-simplex is given by
\begin{enumerate}[label=\arabic*.]
	\item an $n$-simplex $\sigma: [n] \to C$ of the nerve of $C$,
	\item for every nonempty subset $I \subset [n]$, a map of simplicial sets $\Delta^I \to
	F(\sigma(\min(I)))$,
\end{enumerate} 
such that, for every $I \subset I' \subset [n]$, the diagram
\[
\begin{tikzcd}
\Delta^I \ar{r}\ar{d} & F(\sigma(\min(I))) \ar{d}\\
\Delta^{I'} \ar{r} & F(\sigma(\min(I'))).
\end{tikzcd}
\]
commutes. In the case when $F$ takes values in $\Cat \subset \Cat_{\infty}$, fully faithfully
embedded via the nerve, then $\chi(F)$ will be the nerve of the classical Grothendieck construction.
It is proven in \cite[\S 3.2]{LurieHTT}, where $\chi(F)$ is called the {\em relative nerve of $F$},\glsadd{chi1} that
this construction extends to a Quillen equivalence
\[
\phi: (\msSet)_{/\N(C)} \llra \Fun(C^{\op},\msSet): \chi
\]
which, upon passage to $\infty$-categorical localizations, identifies the $\infty$-categories of
\begin{enumerate}[label=(\Roman*)]
	\item functors between the $\infty$-category $\N(C^{\op})$ and the $\infty$-category $\ECat_{\infty}$ of small
	$\infty$-categories,
	\item Cartesian fibrations over $\N(C)$.
\end{enumerate}
In fact, Lurie \cite{LurieHTT} proves a much more sophisticated variant of this equivalence,
governed by certain constructions called straightening and unstraightening, where $\N(C)$ can be replaced by an
arbitrary $\infty$-category $\C$.\\

\noindent
{\bf The relative $2$-nerve.} 
In the present work, we provide a generalization of the relative nerve to
accommodate functors
\[
F: \CC^{(\op,\op)} \lra \Cat_{\infty}
\]
where $\CC$ is a $\Cat$-enriched category and the functor $F$ is $\Cat_{\infty}$-enriched (leaving
the inclusion $\Cat \subset \Cat_{\infty}$ via the nerve implicit). 
To such a functor $F$, we associate a simplicial set $\cchi(F)$, equipped with a forgetful map $\pi:
\cchi(F) \to \Nsc(\CC)$ to the scaled nerve of $\CC$, the latter being a $2$-categorical version of
the Street nerve as introduced in \cite{street}. This {\em relative $2$-nerve} construction $\cchi$
is an adaptation of $\chi$ to the given $2$-categorical context; for example, a $2$-simplex in
$\cchi(F)$ consists of
\begin{enumerate}[label=\arabic*.]
	\item a $2$-simplex $\sigma: \Delta^2 \to \Nsc(\CC)$ of the scaled nerve of $\CC$ whose data we
	label by
	\[
	\begin{tikzcd}[sep=2em]
	&|[alias=X]|{1}\arrow{dr}{12}   &  \\
	0 \arrow[""{name=foo},swap]{rr}{02}\arrow{ur}{01}\arrow[Rightarrow, from=foo,
	to=X, shorten <=.3cm,shorten >=.3cm, "\alpha"] & & {2}
	\end{tikzcd}
	\]
	\item vertices $x_0$,$x_1$, and $x_2$ in the $\infty$-categories $F(0)$, $F(1)$, and $F(2)$, respectively, 
	\item edges 
	\begin{itemize}
		\item $f_{01}: x_0 \to F(01)(x_1)$ in $F(0)$,
		\item $f_{02}:x_0 \to F(02)(x_2)$ in $F(0)$,
		\item $f_{12}:x_1 \to F(12)(x_2)$ in $F(1)$, 
		\item $f_{12 \circ 01}:x_0 \to F(12 \circ 01)(x_2)$ in $F(0)$, 
	\end{itemize}
	\item a $3$-simplex 
	\begin{center}
		\includegraphics{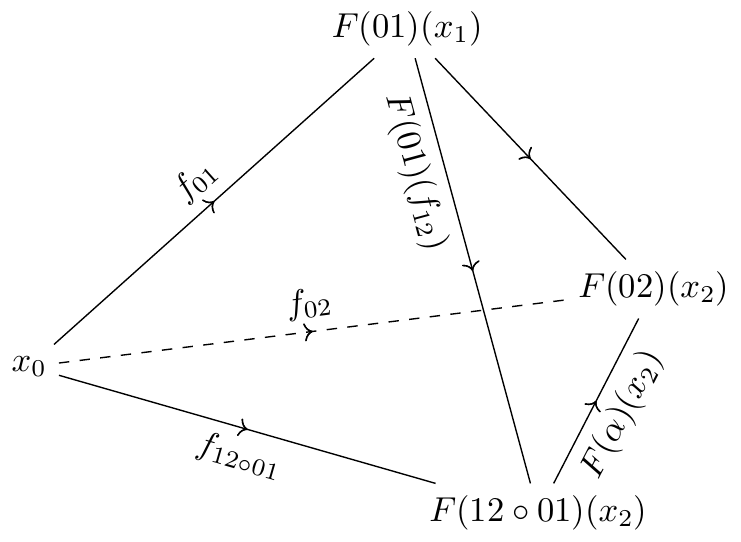}
	\end{center}
	in the $\infty$-category $F(0)$.
\end{enumerate}
The description of the higher dimensional $n$-simplices of $\cchi(F)$ requires some notational
preparation and will be given in \S \ref{sec:const}. Pictorially, just as in the case $n=2$, they
correspond to diagrams parametrized by blown-up versions the standard $n$-simplices (cf. Figure
\ref{fig:posets}). 

Our main result concerning the relative $2$-nerve $\cchi$ is the following:

\begin{introthm}\label{introthm:1}
	The functor $\cchi$ extends to a Quillen equivalence 
	\[
		\PPhi: (\msSet)_{/\Nsc(\CC)} \llra \Fun_{\msSet}(\CC^{(\op,\op)}, \msSet): \cchi
	\]
	modelling an equivalence between $\infty$-categories of
	\begin{enumerate}[label=(\Roman*)]
		\item\label{I} functors between the $\infty$-bicategory $\Nsc(\CC^{(\op,\op)})$ and the
		$\infty$-bicategory $\FCat_{\infty}$ of small $\infty$-categories,
		\item\label{II} locally Cartesian fibrations over $\Nsc(\CC)$ that are Cartesian
		over every scaled triangle.
	\end{enumerate}
\end{introthm}

We emphasize that a much more general version of the equivalence between \ref{I} and \ref{II} in
Theorem \ref{introthm:1}, where $\Nsc(\CC)$ is allowed to be an arbitrary $\infty$-bicategory, has been
proven by Lurie in \cite{LurieGoodwillie} using scaled variants of the straightening and
unstraightening constructions from \cite{LurieHTT}. In fact, our proof of Theorem \ref{introthm:1}
reduces the statement to Lurie's equivalence by means of the following comparison result:

\begin{introthm}\label{introthm:2}
	The functor $\PPhi$ is weakly equivalent to a contravariant version of Lurie's scaled
	straightening functor.
\end{introthm}

The benefit of the functor $\cchi$ over the rather unwieldy scaled unstraightening functor is its
comparatively simple and intuitively clear form which makes it more tractable in explicit
computations. For example, it has already been utilized for precisely this reason in
\cite{DyckerhoffDoldKan}, and we expect it to become a useful tool in controlling $\infty$-bicategorical
limits and colimits. 

We conclude the paper with a comparison of the relative $2$-nerve and the ordinary relative nerve. 

\begin{introthm}\label{introthm:3}
	Let $C$ be an ordinary category. Then, for every functor $F: C^{\op} \to \Cat_{\infty}$,
	there is an explicitly defined equivalence 
	\[
		\chi(F) \overset{\simeq}{\lra} \cchi(F)
	\]
	of Cartesian fibrations over $\N(C)$, natural in $F$.
\end{introthm}

We note that a variant of the Grothendieck construction for $\Cat$-enriched functors $\CC
\to \Cat$ has been introduced and related to the scaled unstraightening functor in \cite{Harpaz}.
Passing through appropriate $\op$'s, our posets $D^I$ correspond to the categories
$|\Delta\!\!\!\!\Delta^n_{/i}|_1$ appearing in \cite{Harpaz} so that our functor $\cchi_{\CC}$ can be
regarded as a direct generalization.

\subsection*{Acknowledgements}

We are grateful to an anonymous referee for the very careful reading which helped improve the
article. T.D. and F.A.G. acknowledge the support of the VolkswagenStiftung through the Lichtenberg
Professorship Programme. The research of T.D. is further supported by the Deutsche
Forschungsgemeinschaft under Germany‘s Excellence Strategy – EXC 2121 „Quantum Universe“ –
390833306. 

\subsection*{Structure of the paper}

In \S \ref{sec:prelim} we provide a brief review of the necessary background on marked simplicial
sets and scaled simplicial sets, the model structures which play a role in the paper, as well as the
scaled (un)straightening constructions. We then define the relative $2$-nerve and attendant
combinatorial constructions in \S \ref{sec:const}. 
The proof of Theorem \ref{introthm:1} occupies most of the rest of the paper and it is achieved by
means of a comparison map with Lurie's straightening functor. The technical heart of the paper
consists of \S \ref{sec:fibinn} where we address the key ingredients necessary for the proof of our
main theorem. We show that the left adjoint to the relative 2-nerve maps inner horn inclusions to
projective trivial cofibrations, establishing a weak equivalence between said adjoint and Lurie's
scaled straightening functor. 

The subsequent arguments in \S \ref{sec:qe} are then comparatively routine: We show Theorem
\ref{introthm:2}, obtaining as a corollary the full statement of Theorem \ref{introthm:1}. We
conclude in \S \ref{sec:1nerve} with the comparison of the $1$-categorical and $2$-categorical
relative nerves as stated in Theorem \ref{introthm:3}.

\subsection*{Glossary}

For ease of reading, we provide a (non-comprehensive) list of notation appearing in this paper, as
well as references to the relevant pages.

\printglossaries

\section{Preliminaries}\label{sec:prelim}

\subsection{$2$-categories}

We denote by $\Cat$ the $1$-category of small categories. We treat $\Cat$ as a symmetric monoidal
category by virtue of the Cartesian product and define a $2$-category to be a $\Cat$-enriched
category. Further, a $2$-functor between $2$-categories is defined to be a $\Cat$-enriched functor. 

Since a $2$-category has both $1$- and $2$-morphisms, there are two different ways of taking
opposite categories. For a $2$-category $\CC$, we denote by $\CC^{(\op,-)}$ the $2$-category in
which the directions of the $1$-morphisms have been reversed, and by $\CC^{(-,\op)}$ the
$2$-category in which the directions of the $2$-morphisms have been reversed. Clearly,
$(\CC^{(\op,-)})^{(-,\op)}\cong (\CC^{(-,\op)})^{(\op,-)}$, and we will denote this $2$-category by
$\CC^{(\op,\op)}$.   

Given a $2$-category $\CC$ and an object $c \in \CC$, we define its \emph{over $2$-category} $\CC_{c/}$ to have: 
\begin{itemize}
	\item Objects given by morphisms $f:c\to d$ in $\CC$. 
	\item $1$-morphisms from $f: c\to d$ to $f': c\to d'$ given by commutative diagrams
	\[
	\begin{tikzcd}
	&c\arrow[dl, "f"']\arrow[dr, "f^\prime", ""{name=D,inner sep=1pt}] & \\
	|[alias=Z]| d\arrow[rr, "g"'] \arrow[shorten <=.9cm,shorten >=.5cm, Rightarrow, from=Z, to=D, "\alpha"]& & d^\prime 
	\end{tikzcd}
	\]
	\item $2$-morphisms from $(g,\alpha)$ to $(h,\beta)$ given by a $2$-morphism
		$\mu:g \Rightarrow h$ in $\CC$ satisfying $\alpha = \beta \circ (\mu *
		\id_f)$.
\end{itemize}
We moreover define the \emph{homotopy category} $|\mathbb{C}|_1$ of a $2$-category $\CC$ to be the ordinary
$1$-category with 
\begin{itemize}
	\item the same objects as $\mathbb{C}$,
	\item for objects $c,c' \in \mathbb{C}$, the set of morphisms
	\[
	|\mathbb{C}|_1(c,c')=\pi_0 |\mathbb{C}(c,c')|,
	\]
	where $|-|$ denotes the geometric realization of the nerve.
\end{itemize} 

\subsection{Marked simplicial sets}
\label{subsec:marked}

Following \cite{LurieHTT}, we define a \emph{marked simplicial set} to be a simplicial set $X$ together
with a chosen subset $E \subset X_1$ of {\em marked edges} containing all degenerate edges. We will sometimes use the notation $\overline{X}=(X,E)$ for a marked simplicial set. We denote the category
of marked simplicial sets by $\msSet$. For a simplicial set $K$, we denote by $K^{\flat}$ (resp.
$K^{\sharp}$) the marked simplicial set with marked edges given by the degenerate edges (resp. all
edges). The category $\msSet$ is Cartesian closed so that, for every pair of objects $X,Y$ in
$\msSet$, there exists an internal mapping object which we denote by $\Map(X,Y)$. There are
two related ways to provide $\sSet$-enrichments of $\msSet$:
\begin{enumerate}[label=(\arabic*)]
	\item The underlying simplicial set $\Map^\flat(X,Y)$ of $\Map(X,Y)$ provides a
		simplicial enrichment adjoint to the tensor structure
		\[
			(K,\overline{X})\mapsto K^\flat\times \overline{X}. 
		\]
		of $\msSet$ over $\sSet$.
	\item The simplicial subset $\Map^{\sharp}(X,Y) \subset \Map^\flat(X,Y)$ consisting
		of those simplices whose edges are marked in $\Map(X,Y)$ provides a
		simplicial enrichment adjoint to the tensor structure
		\[
			(K,\overline{X})\mapsto K^\sharp \times \overline{X}
		\]
		of $\msSet$ over $\sSet$.
\end{enumerate}
Given a $2$-category $\CC$, we may consider $\CC$ as a $\msSet$-enriched category by virtue of
taking nerves of the mapping $1$-categories of $\CC$ and marking those edges corresponding to
$2$-isomorphisms. In particular, we will speak of $\msSet$-enriched functors between a $2$-category
and a $\msSet$-enriched category.

\subsection{Scaled simplicial sets}

Scaled simplicial sets provide a model for a homotopy coherent theory of
$(\infty,2)$-categories, called the theory of $\infty$-bicategories, paralleling the way
ordinary simplicial sets model a homotopy coherent theory of $(\infty,1)$-categories. The
idea goes back to work of Street \cite{street} where the notion of a nerve of a strict
$n$-category is defined as a simplicial set equipped with additional data that keeps
track of the invertibility of higher morphisms. Verity's {\em complicial sets} \cite{verityI,verityII}
are then obtained by weakening the horn filling properties of the Street nerve, leading to a
candidate model for $(\infty,n)$-categories. Along the same lines, Lurie
\cite{LurieGoodwillie} introduces scaled simplicial sets and proves that they indeed provide
a model for $(\infty,2)$-categories equivalent to the other known models.

A {\em scaled simplicial set} consists of a simplicial set $K$ together with a chosen subset $T
\subset K_2$ of $2$-simplices, containing all degenerate $2$-simplices. The elements of $T$ are
called {\em thin}. Given a scaled simplicial set $K$, we denote by $K^{\th} \subset K$ the simplicial subset of $K$
consisting of those simplices all of whose $2$-simplices are thin. For a simplicial set $K$, we
denote by $K_{\flat}$ the scaled simplicial set whose only thin simplices are the degenerate
$2$-simplices and by $K_{\sharp}$ the scaled simplicial set with all $2$-simplices thin.  

Given a scaled simplicial set $(K,T)$, a marked simplicial set $(X,M)$ and a map
$p:X\to K$, we say that $p$ is a \emph{scaled Cartesian fibration} if 
\begin{enumerate}[label=(\arabic*)]
	\item $p$ is a locally Cartesian fibration in the sense of \cite[2.4.2.6]{LurieHTT},
	\item $M$ is the set of locally $p$-Cartesian edges of $X$, and 
	\item The restriction of $p$ to every thin $2$-simplex $\Delta^2 \to K^{\th}$ is a Cartesian fibration. 
\end{enumerate}

\subsection{Model structures}
\label{subsec:model}

We recall certain model structures introduced in \cite{LurieHTT} and \cite{LurieGoodwillie}
which will be relevant for us, as well as some facts about widely-known model structures
which will be of use in our coming arguments. Of particular importance are the model
structures on the categories $\msSet$ and $(\msSet)_{/K}$, which are integral to the Quillen
adjunction at the heart of this paper. 

The basic building block for our construction will be the model structure on
$(\msSet)_{/K}$ which models Cartesian fibrations, cf. \cite[Prop. 3.1.3.7]{LurieHTT}: 

\begin{theorem}
	Let $K$ be a simplicial set. There exists a left-proper combinatorial simplicial model structure on $(\msSet)_{/K}$ with 
	\begin{enumerate}
		\item cofibrations those morphisms whose underlying maps of simplicial sets are monomorphisms,
		\item equivalences the Cartesian equivalences of \cite[Prop.
			3.1.3.3]{LurieHTT},
		\item fibrant objects $p:\overline{X}\to K$, where $p$ is a Cartesian fibration 
			with $p$-Cartesian edges marked,
		\item for objects $X$ and $Y$, the simplicial enrichment is given by the
			simplicial set $\Map^\sharp(X,Y)$ from \S \ref{subsec:marked}. 
	\end{enumerate}
\end{theorem}

We will refer to this model structure as the \emph{Cartesian model structure}

\begin{remark}
	A few facts about the Cartesian model structure will be of use in the sequel:
	\begin{enumerate}
		\item There is a dual model structure on $(\msSet)_{/K^{\op}}$, the \emph{coCartesian model structure}. 
		\item As a special case of the Cartesian model structure, we obtain a model
			structure on $\msSet\cong (\msSet)_{/\ast}$, whose fibrant objects
			are naturally marked $\infty$-categories. We will also refer to this
			model structure as the \emph{marked model structure}.
	\end{enumerate}
\end{remark}

For the right-hand side of the Grothendieck construction, we will need a model structure
which models $(\infty,2)$-functors into the $(\infty,2)$-category $\FCat_{\infty}$. This is
provided by the \emph{projective model structure} on the category $\Fun_{\msSet}(\CC,
\msSet)$ of $\msSet$-enriched functors with $\msSet$-enriched natural transformations: 

\begin{theorem}
	For every 2-category $\CC$, the category $\Fun_{\msSet}(\CC, \msSet)$ carries a
	combinatorial left-proper model structure in which a natural transformation
	$\alpha:F\to G$ is 
	\begin{enumerate}
		\item a fibration if the induced map $\alpha_C:F(C)\to G(C)$ is a fibration
			in the marked model structure,
		\item a weak equivalence if the induced map $\alpha_C:F(C)\to G(C)$ is a
			weak equivalence in the marked model structure. 
	\end{enumerate} 
\end{theorem}

\begin{proof}
	See, e.g., \cite[Prop. A.3.3.2]{LurieHTT}. 
\end{proof}

\begin{remark}
	The projective model structure on $\Fun_{\msSet}(\CC, \msSet)$, equipped with mapping spaces
	$\Map^\sharp(-,-)$, is simplicial \cite[Remark 3.8.2]{LurieGoodwillie}.
\end{remark}

Finally, we need a model structure for the left-hand side of the Grothendieck construction,
modelling scaled Cartesian fibrations over $\Nsc(\CC)$. 

\begin{theorem}
	Let $(K,T)$ be a scaled simplicial set. There exists a left-proper, combinatorial,
	simplicial model structure on $(\msSet)_{/K}$ such that
	\begin{enumerate}
		\item A morphism is a cofibration if and only if it induces a monomorphism
			between underlying simplicial sets.
		\item A object $p:\overline{X}\to K$ is fibrant if and only if $p$ is a
			scaled Cartesian fibration, and the marked edges are precisely the
			locally $p$-Cartesian edges.  
	\end{enumerate} 
\end{theorem}

\begin{proof}
	This is (the dual of) a special case of \cite[Thm 3.2.6]{LurieGoodwillie}. 
\end{proof}

\subsection{Scaled straightening and unstraightening}

Given a scaled simplicial set $\overline{K}=(K,T)$ and a weak equivalence
$\phi:\mathfrak{C}^{\on{sc}}[(K,T)]\to \C$ of marked simplicially enriched categories, the
$(\infty,2)$-Grothendieck construction is presented by a Quillen equivalence 
\[
\CStsc_{\phi}:(\msSet)_{/K} \llra \Fun(\C,\msSet): \CUnsc_{\phi}
\]
constructed in \cite[Theorem 3.8.1]{LurieGoodwillie}. The left and right adjoints are called, respectively,
scaled straightening and scaled unstraightening. As proved in \cite{LurieGoodwillie}, this Quillen
equivalence relates the scaled \emph{co}Cartesian model structure to the projective model
structure.\footnote{The functors we denote by $\CStsc$ and $\CUnsc$ appear in \cite{LurieGoodwillie}
as $\Stsc$ and $\Unsc$. We will, however, use the latter symbols for the corresponding functors in
the \emph{contravariant} Grothendieck construction.} In this paper, we consider the special case
where $\CC$ is a 2-category and $\phi:\mathfrak{C}^{\on{sc}}[\Nsc(\CC)]\to \CC$ is the counit of the
adjunction $\mathfrak{C}^{\on{sc}}\vdash \Nsc$ (cf. \cite[Theorem 4.2.7]{LurieGoodwillie}). In this context, we will denote the straightening
and unstraightening by $\CStsc_{\CC}$ and $\CUnsc_{\CC}$, respectively.

Since our relative nerve construction produces a scaled \emph{Cartesian} fibration, it will relate
to the duals of Lurie's scaled straightening and unstraightening functors. More precisely, it will
related to the Quillen adjunction given by the \emph{contravariant scaled straightening} $\St_{\CC}$
and the \emph{contravariant scaled unstraightening} $\Un_{\CC}$, given respectively by the
composites 
\begin{equation}\label{eq:CSTcomposite}
\begin{tikzcd}
(\msSet)_{/\Nsc(\CC)}\arrow[r,"\op"] & (\msSet)_{/\Nsc(\CC)^{\op}}\arrow[r, "\operatorname{CSt}^{\operatorname{sc}}"] & \mathsf{Fun}(\CC^{(\op,-)}, \msSet)\arrow[r,"\op\circ-"] & \mathsf{Fun}(\CC^{(\op,\op)}, \msSet)
\end{tikzcd}
\end{equation}
and 
\begin{equation}
\begin{tikzcd}
\mathsf{Fun}(\CC^{(\op,\op)}, \msSet)\arrow[r,"\op\circ-"]  & \mathsf{Fun}(\CC^{(\op,-)}, \msSet)\arrow[r,"\operatorname{CUn}^{\operatorname{sc}}"] & (\msSet)_{/\Nsc(\CC)^{\op}}\arrow[r,"\op"] & 	(\msSet)_{/\Nsc(\CC)}.
\end{tikzcd}
\end{equation}
The aim of this section is to provide an explicit description of the former functor. 

Before we give this description, however, it is worth pausing to comment the appearance of $\op$'s in the above composites. Given marked simplicial sets $\overline{X}$ and $\overline{Y}$, there is a bijection 
\[
	\Hom_{\msSet}(\overline{X},\overline{Y})\cong \Hom_{\msSet}(\overline{X}^{\op},\overline{Y}^{\op})
\]
However, this does \emph{not} extend to a simplicially-enriched functor 
\[
\op:\msSet \to \msSet; \quad 
X \mapsto X^{\op}.
\] 
This is because $\Map^\flat(X,Y)\not\cong\Map^\flat(X^\op,Y^\op)$, but instead
\[
\Map^\flat(X,Y)\cong\Map^\flat(X^\op,Y^\op)^\op.
\] 
Consequently, we need to apply $\op$ to all of the Hom-spaces of $\msSet$ to get an enriched functor 
\[
\op:\msSet \to (\msSet)^{(-,\op)};\quad 
X \mapsto X^\op.
\] 
This accounts for the removal of the second $\op$ from $\Nsc(\CC)$ in the composite above.

We now proceed with the definition of $\St_{\CC}$ obtained by the following transcription of the constructions in
\cite[\S 3.5]{LurieGoodwillie}.
\begin{construction}
	For $X$ a simplicial set, we define the \emph{right cone over $X$} to be 
	\[
	C^R(X):=X\times \Delta^1\coprod_{X\times\{1\}}\Delta^0.
	\]
	Similarly we set
	\[
	\mathfrak{C}^R(X):=\mathfrak{C}[C^R(X)].
	\]
	For $X\to K$ a morphism of simplicial sets, we define relative variants of both of these cones via, e.g., 
	\[
	\mathfrak{C}^R_K(X):= \mathfrak{C}^R(X)\coprod_{\mathfrak{C}[X\times \{0\}]}\mathfrak{C}[K].
	\]
	
	Now suppose that $\overline{X}=(X,M)$ is a marked simplicial set and $\overline{K}=(K,T)$ is a scaled simplicial set, with $X\to K$ a morphism. We can then define a scaling $T_R$ on $X\times \Delta^1$ by requiring that
	\begin{itemize}
		\item All degenerate 2-simplices are in $T_R$. 
		\item For $\sigma:\Delta^2\to X\times \Delta^1$, if the preimage of $1$ under the composite 
		\[
		\begin{tikzcd}
		\Delta^2\arrow[r, "\sigma"] & X\times \Delta^1\arrow[r] & \Delta^1
		\end{tikzcd}
		\] 
		is $\{1,2\}$ and $\sigma|_{\Delta^{\{1,2\}}}\in M$ then $\sigma\in T_R$.  
	\end{itemize}
	We then define a scaled variant of the cone, 
	\[
	\overline{C}^R(\overline{X}):=(X\times \Delta^1,T_R)\coprod_{(X\times\{1\})_\flat}(\Delta^0)_\flat.
	\]
	This, together with the scaling on $K$ yields a scaled relative cone $\overline{C}^R_{\overline{K}}(\overline{X})$. We then define marked simplicial categories  $\overline{\mathfrak{C}}^R(\overline{X})$ and $\overline{\mathfrak{C}}^R_{\overline{K}}(\overline{X})$ by applying the scaled rigidification $\mathfrak{C}^{\operatorname{sc}}$ to the respective scaled cones. 
\end{construction}

\begin{definition}\label{defn:CSt}
	Let $\CC$ be a 2-category and $\overline{X}\to \Nsc(\CC)$ a marked simplicial set over $\Nsc(\CC)$. We define 
	\[
	\CC(X):= \CC\coprod_{\mathfrak{C}^{\on{sc}}[\Nsc(\CC)]}\overline{\mathfrak{C}}^R_{\Nsc(\CC)}(\overline{X}).
	\]
	The \emph{contravariant scaled straightening functor}\glsadd{Stsc} $\St_{\CC}$
	is the functor which sends $\overline{X}\to \Nsc(\CC)$ to the functor 
	\[
		\CC^{(\op,\op)} \lra \msSet, \quad x\mapsto \Map_{\CC(X)}(x,v)^{\op}.
	\]
\end{definition}

It is straightforward to check that $\St_{\CC}$ as defined in \eqref{eq:CSTcomposite} is indeed naturally isomorphic to $\St_{\CC}$ as defined in Definition \ref{defn:CSt}.

\section{The relative $2$-nerve}\label{sec:const}

In this section, we define the central notion of this work: the relative $2$-nerve. We begin with the definition of
the scaled nerve. 

\begin{definition}\label{defn:OI}
	Let $I$ be a linearly ordered finite set. We define a $2$-category $\OO^{I}$\glsadd{OI} as
	follows\footnote{This notation is a deliberate nod to the \emph{orientals} defined by Street
	in \cite{street}. In the terminology of this work, the $2$-category $\mathbb{O}^{[n]}$ is
	the free $2$-category on the $n$-simplex.}:
	\begin{itemize}
		\item the objects of $\OO^I$ are the elements of $I$,
		\item the category $\mathbb{O}^{I}(i,j)$ of morphisms between objects $i,j \in I$ is defined
		as the poset of finite sets $S \subseteq I$ such that $\min(S)=i$ and $\max(S)=j$
		ordered by inclusion,
		\item the composition functors are given, for $i,j,l\in I$, by
		\[
		\mathbb{O}^{I}(i,j) \times \mathbb{O}^{I}(j,l) \to \mathbb{O}^{I}(i,l), \; (S,T) \mapsto S \cup T.
		\]
	\end{itemize}
\end{definition}

\begin{remark} The association $[n] \mapsto \OO^{[n]}$ extends to a cosimplicial $2$-category so
	that, for a $2$-category $\CC$, we obtain a simplicial set $\Nsc(\CC)$ with set of
	$n$-simplices
	\[
	\Nsc(\CC)_n = \Fun_{\Cat}(\OO^{[n]}, \CC).
	\]
	The simplicial set $\Nsc(\CC)$ is precisely the scaled nerve of \cite{LurieGoodwillie} where the
	scaled $2$-simplices are those for which the $2$-morphism $\{0,2\} \Rightarrow \{0,1,2\}$
	in $\OO^{[2]}$ is mapped to an invertible $2$-morphism in $\CC$.
\end{remark}

For a linearly ordered set $I$, we consider the under $2$-category
\begin{equation}\label{defn:GI}
\tilde{D}^I = \left((\mathbb{O}^{I})^{(\mathblank,\text{op})}\right)_{ \min (I) \big/}
\end{equation}
and denote its homotopy category by $D^I$.\glsadd{DI} For the standard ordinal $[n]=\{0,1,\ldots,n\}$, we will
use the notation $D^n = D^{[n]}$. 

\begin{proposition}\label{prop:GIposet}
	The category $D^{I}$ can be identified with the poset whose
	\begin{itemize}
		\item elements are those subsets $S \subseteq I$ such that $\min(S)=\min(I)$, and
		\item we have $S \leq T$ if and only if the following conditions hold:
		\begin{enumerate}[label=(\arabic*)]
			\item $\max(S) \leq \max(T)$,
			\item $T \subset S \cup [\max(S),\max(T)]$.
		\end{enumerate}
	\end{itemize}
\end{proposition}
\begin{proof}
	We note that, for $S,T \in \tilde{D}^{I}$, a $1$-morphism in $\tilde{D}^{I}$ from $S$ to $T$ corresponds to a
	subset $U \subset I$ satisfying $\min(U)=\max(S)$, $\max(U)=\max(T)$, and $T \subseteq S \cup U$.
	In particular, we have
	\[
	T \subseteq S \cup [\max(S),\max(T)]
	\]
	so that $S$ must contain all elements $t \in T$ with $t \le \max(S)$, implying condition (2).
	Vice versa, if a pair $(S,T)$ satisfies the conditions (1) and (2), then the set $U_{0} :=
	[\max(S),\max(T)]$ defines an $1$-morphism in $\tilde{D}^{I}$ from $S$ to $T$. Further, given any
	$1$-morphism in $\tilde{D}^{I}$ from $S$ to $T$, represented by a subset $U$, we have an inclusion $U
	\subset U_0$, which represents a $2$-morphism in $\tilde{D}^{I}$ from $U_0$ to $U$. This shows that all
	nonempty morphism categories of $\tilde{D}^{I}$ are connected (even contractible) and implies the claim.
\end{proof}

A systematic analysis of the poset $D^n$ will be given in \S \ref{subsec:innerhorns}. Graphical
representations for $n \le 4$ are provided in Figure \ref{fig:posets}.
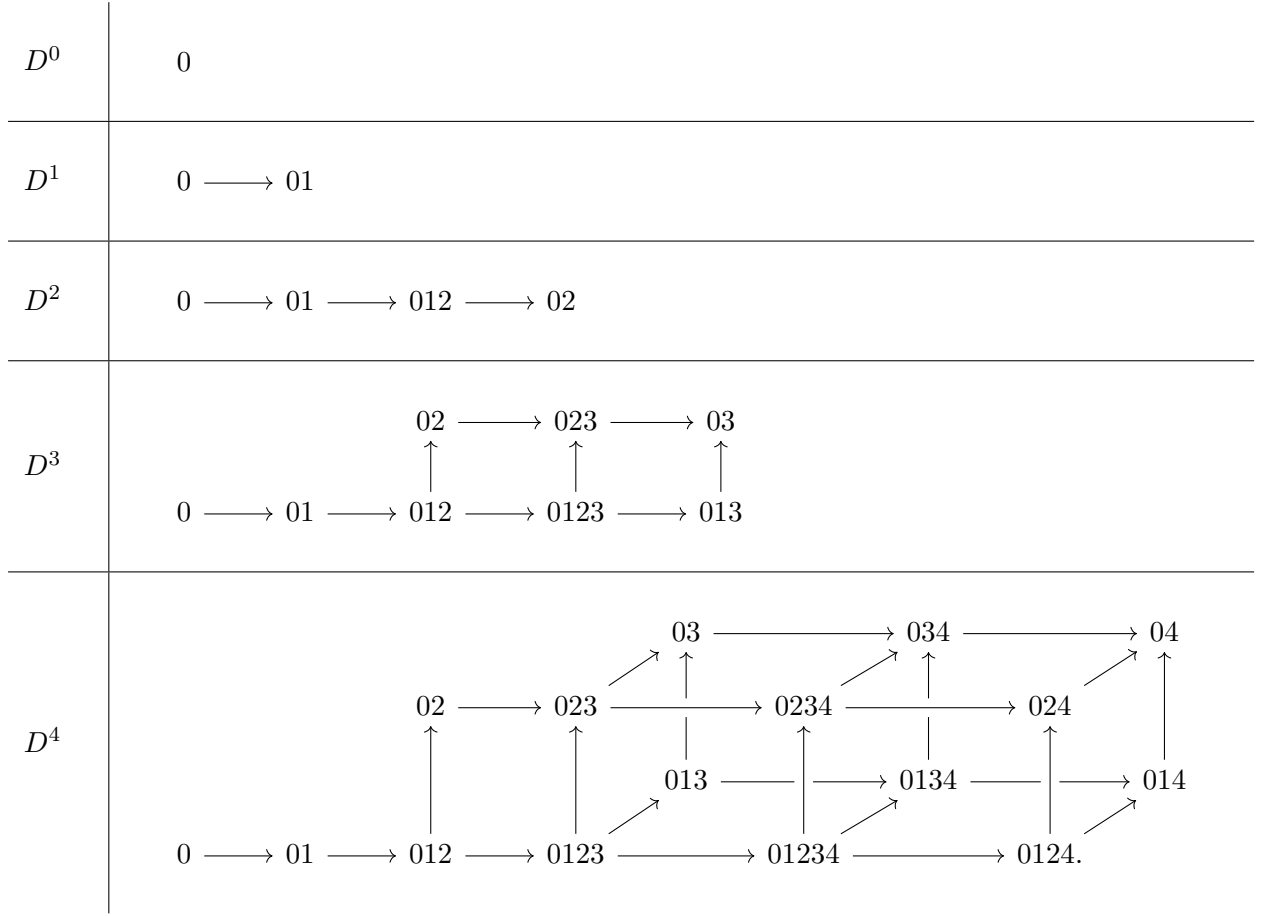
\begin{figure}[h]
	\centering
	\begin{tabular}{l l|l}
		$D^0$ & & 
		\begin{tikzpicture}[baseline={([yshift=-.7ex]current bounding box.center)},mybox/.style={ inner sep=15pt}]
\node[mybox] (box){%
		\begin{tikzcd}0 \end{tikzcd}
	};
	\end{tikzpicture}\\\hline
		$D^1$ & & 
		\begin{tikzpicture}[baseline={([yshift=-.7ex]current bounding box.center)},mybox/.style={ inner sep=15pt}]
\node[mybox] (box){%
		\begin{tikzcd}0 \ar{r} & 01 \end{tikzcd}
	};
	\end{tikzpicture}\\\hline
		$D^2$ & & 
		\begin{tikzpicture}[baseline={([yshift=-.7ex]current bounding box.center)},mybox/.style={ inner sep=15pt}]
\node[mybox] (box){%
		\begin{tikzcd}0 \ar{r} & 01 \ar{r} & 012 \ar{r} & 02\end{tikzcd}
	};
	\end{tikzpicture}\\\hline
		$D^3$ &  &
	\begin{tikzpicture}[baseline={([yshift=-.7ex]current bounding box.center)},mybox/.style={ inner sep=15pt}]
\node[mybox] (box){%
		\begin{tikzcd} & & 02\arrow[r] & 023\arrow[r] &03 \\
	0 \arrow[r]& 01\arrow[r] & 012\arrow[r]\arrow[u] & 0123\arrow[r]\arrow[u] & 013 \arrow[u]
	\end{tikzcd}
	};
	\end{tikzpicture}\\\hline
	$D^4$ & &
	\begin{tikzpicture}[baseline={([yshift=-.7ex]current bounding box.center)},mybox/.style={ inner sep=15pt}]
\node[mybox] (box){%
	\begin{tikzcd}[row sep=1.2em,column sep=1.2em]
	& &&  && && 03\arrow[rr]\ar[dd,leftarrow] && 034\arrow[rr]\ar[dd,leftarrow]&& 04\\
	&& && 02\arrow[rr] && 023\arrow[rr, crossing over]\arrow[ur] && 0234\arrow[rr, crossing over]\arrow[ur] && 024\arrow[ur]\\
	& && && && 013\arrow[rr] && 0134\arrow[rr] && 014\arrow[uu]\\
	0\arrow[rr] && 01\arrow[rr]&& 012\arrow[rr]\arrow[uu] && 0123\arrow[rr]\arrow[uu]\arrow[ur] &&
	01234\arrow[rr]\arrow[uu, crossing over]\arrow[ur] && 0124.\arrow[uu, crossing over]\arrow[ur]
	\end{tikzcd}
};
\end{tikzpicture}
	\end{tabular}
	\caption{The posets $D^{n}$ for $n \le 4$.}
	\label{fig:posets}
\end{figure}
Let $I$ be a finite nonempty linearly ordered set and suppose $J \subset I$ is a nonempty subset.
Then pullback along $1$-morphisms in $\OO^I$ determines a {\em pullback functor}\glsadd{rho}
\begin{equation}\label{eq:pullbackO}
\rho_{J,I}:\; \OO^{I}(\min(I),\min(J))^{\op} \times D^{J} \to D^{I},\; (S_1,S_2) \mapsto S_1 \cup S_2.
\end{equation}

\begin{definition}\label{defi:rel2nerve}	
	Let $\CC$ be a $2$-category and let 
	\[
	F: \CC^{(\op,\op)} \lra \msSet
	\]
	be a $\msSet$-enriched functor. 
	We define a marked simplicial set $\cchi_{\CC}(F)$, called the {\em relative $2$-nerve of $F$}, as follows.
	An $n$-simplex of $\cchi_{\CC}(F)$ \glsadd{chi2} consists of 
	\begin{enumerate}
		\item an $n$-simplex $\sigma: \Delta^n \to \Nsc(\CC)$,
		\item for every nonempty subset $I \subset [n]$, a map of marked simplicial sets
		\[
		\theta_I: \N(D^{I})^{\flat} \to  F(\sigma(\min(I))),
		\]
	\end{enumerate}
	such that, for every $J \subset I \subset [n]$, the diagram 
	\[
	\begin{tikzcd}
	\N(\OO^{I}(\min(I),\min(J))^{\op})^{\flat} \times \N(D^{J})^{\flat}
	\arrow{r}{\rho_{J,I}} \arrow[d,"\N(\sigma) \times \theta_J"] & 
	\N(D^{I})^{\flat} \arrow[d,"\theta_I"] \\
	\N(\CC(\sigma(\min(I)),\sigma(\min(J)))^{\op})^{\flat} \times F(\min(J))
	\arrow{r}{F(-)} & 
	F(\min(I))
	\end{tikzcd}
	\]
	commutes. The marked edges of $\cchi_{\CC}(F)$ are defined as follows: An edge $e$ of $\cchi_{\CC}(F)$ consists of a
	morphism $f: x \to y$ in $\CC$, together with vertices $A_x \in F(x)$, $A_y \in F(y)$, and
	an edge $\widetilde{e}: A_x \to F(f)(A_y)$ in $F(x)$. We declare $e$ to be marked if
	$\widetilde{e}$ is marked. Finally, we consider $\cchi_{\CC}(F)$ as a simplicial set over
	$\Nsc(\CC)$ by means of the forgetful functor.
\end{definition}

The construction $F \mapsto \cchi_{\CC}(F)$ is functorial with respect to $\msSet$-enriched natural
transformations and therefore defines a functor
\begin{equation}\label{eq:rel2nerve}
\cchi_{\CC}: \Fun_{\msSet}(\CC^{(\op,\op)},\msSet) \lra (\msSet)_{/\Nsc(\CC)}
\end{equation}
which we refer to as the \emph{relative $2$-nerve functor}. 

To conclude this section we observe that $\cchi_{\CC}$ preserves all limits and therefore, invoking
the adjoint functor theorem, we obtain an adjunction
\begin{equation}\label{eq:adjunction}
		\PPhi_{\CC}: (\msSet)_{/\Nsc(\CC)} \llra \Fun_{\msSet}(\CC^{(\op,\op)}, \msSet): \cchi_{\CC}.
\end{equation}
with left adjoint $\PPhi_{\CC}$. 

\section{Key propositions}\label{sec:fibinn}

We will show that the adjunction \eqref{eq:adjunction} is a Quillen adjunction by verifying that
$\cchi_{\CC}$ preserves trivial fibrations and that $\PPhi_{\CC}$ preserves weak equivalences. To this
end, we will establish the two following results whose proofs form the technical heart of this work.

\begin{proposition}\label{prop:trfib}
	Let $\CC$ be a $2$-category. Then $\cchi_{\CC}$ maps trivial fibrations in the projective
	model structure to trivial fibrations in the scaled Cartesian model structure.
\end{proposition}

\begin{proposition}\label{prop:innhorn}
	Let $\CC=\mathbb{O}^{n}$ with $n\geq 2$ and consider the inclusion morphism,
	\[
		(\Lambda^{n}_i)^{\flat} \to (\Delta^{n})^{\flat} \enspace 0<i<n,
	\]
	as a morphism in $(\msSet)_{/\Nsc(\OO^n)}$. Then the induced map,
	\[
		\PPhi_{\OO^n}((\Lambda^{n}_i)^{\flat}) \to \PPhi_{\OO^n}((\Delta^{n})^{\flat})
	\]
	is a trivial cofibration in the projective model structure.
\end{proposition}

Proposition \ref{prop:trfib} implies that $\PPhi_{\CC}$ preserves cofibrations. Proposition
\ref{prop:innhorn} will be the main technical ingredient to show that show that $\PPhi_{\CC}$ is
weakly equivalent to $\St_{\CC}$ and hence preserves weak equivalences. Moreover, $\PPhi_{\CC}$ will
then induce an equivalence on derived categories (since $\St_\CC$ does), and thus will be a Quillen
equivalence. 

\subsection{Proposition \ref{prop:trfib}}\label{subsec:trfib}

Throughout this section, we fix a $2$-category $\CC$. We want to show that $\cchi_{\CC}$ preserves trivial fibrations. Therefore we need to consider lifting problems against arbitrary cofibrations of marked simplicial sets. These are generated by 
\[
\left\lbrace(\Delta^1)^\flat\hookrightarrow (\Delta^1)^\sharp\right\rbrace
\]
and 
\[
\left\lbrace(\partial\Delta^n)^\flat\hookrightarrow (\Delta^n)^\flat\right\rbrace.
\]
Consequently, we will confine ourselves to these cases and find solutions to lifting problems of the form 
\begin{equation}\label{eq:red1lifting}
\begin{tikzcd}
(\Delta^1)^\flat \ar{r}\ar{d} & \cchi_{\CC}(F) \ar{d}\\
(\Delta^1)^\sharp \ar{r}\ar[dashed]{ur} & \cchi_{\CC}(G).
\end{tikzcd}
\end{equation}
\begin{equation}\label{eq:red2lifting}
\begin{tikzcd}
(\partial \Delta^n)^{\flat} \ar{r}{f}\ar{d} & \cchi_{\CC}(F) \ar{d}\\
(\Delta^n)^{\flat} \ar{r}{g}\ar[dashed]{ur} & \cchi_{\CC}(G)
\end{tikzcd}
\end{equation}
where $F \to G$ is a $\msSet$-enriched natural transformation of $\msSet$-enriched functors
$\CC^{(\op,\op)} \to \msSet$ which is a pointwise trivial fibration. Here, the map $\cchi_{\CC}(F) \to \cchi_{\CC}(G)$ is considered as a morphism in $\msSet_{/\Nsc(\CC)}$. A quick inspection shows that we can solve the lifting problem \eqref{eq:red1lifting}. The lifting problem given by \eqref{eq:red2lifting} is slightly more involved and requires some preparation.

We begin by developing some terminology that will help unravel the essence of these kinds of lifting
problems. In the following discussion, markings will not play any role so we systematically ignore
them and only refer to the underlying simplicial sets.

Let $n \ge 0$ and let $J \subset [n]$ be a nonempty subset. The pullback functor
\begin{equation}\label{eq:ref}
\rho_{J}^n: \OO^{[n]}(0,\min(J))^{\op} \times D^J \lra D^n
\end{equation}
is easily verified to be fully faithful; we denote its image by $A(J)$.\glsadd{AJ} We denote by $\Dn$ the
nerve of the poset $D^n$ and by $\A(J)$ the nerve of the poset $A(J)$. We can now define,

\begin{equation}\label{eq:sn1}
\Sn := \bigcup_{\Delta^J\subset \partial \Delta^n} \A(J) \subset \Dn.
\end{equation}

Suppose we are given a lifting problem of the form \eqref{eq:red2lifting}. According to Definition
\ref{defi:rel2nerve}, the map $g$ corresponds to a functor $\sigma: \OO^{[n]} \to \CC$
together with a compatible collection of maps 
\[
\Big\{g_I: \D^I \lra G(\sigma(\min(I)))\Big\}
\] 
parametrized by the nonempty subsets $I \subset [n]$. In particular, choosing $I = [n]$, we obtain a
functor
\[
\overline{g}: \Dn \lra G(\sigma(0)).
\] 
The additional data comprised in the map $f$ amounts to a compatible collection of functors
\[
\Big\{ f_J: \D^J \lra F(\sigma(\min(J)))\Big\}
\] 
parametrized by the subsets $J$ such that $\Delta^{J} \subset \partial \Delta^n$. We can use the functoriality of $F$ to transfer the $f_J$ from $\sigma(\min(J))$ to $\sigma(0)$, to obtain a functor as the composite
\[
f^n_J: A(J) \overset{\rho_J^n}{\longleftarrow} \OO^{[n]}(0,\min(J))^{\op} \times \D^J
\overset{\sigma \times f_J}{\lra} \CC(\sigma(0),\sigma(\min(J)))^{\op} \times
F(\sigma(\min(J))) \overset{F}{\lra} F(\sigma(0)).
\] 
The various functors $\{ f_J^n \}$ assemble to define a functor
\[
\overline{f}: \Sn \lra F(\sigma(0))
\]
making the solid part of the diagram 
\begin{equation}\label{eq:redlifting}
\begin{tikzcd}
\Sn \ar{r}{\overline{f}}\ar{d} & F(\sigma(0)) \ar{d}\\
\Dn\ar{r}{\overline{g}}\ar[dashed]{ur} & G(\sigma(0))
\end{tikzcd}
\end{equation}
commute. We call \eqref{eq:redlifting} the {\em reduced lifting problem} associated to \eqref{eq:red2lifting}.

\begin{proposition}\label{prop:reducedlifting} The solutions of a lifting problem of the form \eqref{eq:red2lifting} are in bijection with the solutions of the associated reduced lifting problem \eqref{eq:redlifting}.
\end{proposition}
\begin{proof} 
	Left to the reader.
\end{proof}

Since the morphism $\Sn \to \D^n$ is a cofibration, Proposition \ref{prop:trfib} follows. 

\subsection{Proposition \ref{prop:innhorn}}\label{subsec:innerhorns}

This section is devoted to the key technical claim of this paper. Namely we show that the images of inner horn inclusions under $\PPhi_{\CC}$ are objectwise marked weak equivalences. We begin by identifying the images of $(\Lambda^n_i)^\flat$ and $(\Delta^n)^\flat$ under $\PPhi_{\OO^n}$. 

\begin{definition}
	For $n\geq 0$, we define a marked simplicially enriched functor\glsadd{Dadj}
	\[
	\mathfrak{D}^n: (\OO^n)^{(\op,\op)} \to \msSet; \quad j \mapsto \left(\D^{[j,n]}\right)^\flat. 
	\]
	For $0<i<n$, we define a second marked simplicially enriched functor \glsadd{Ladj}
	\[
	\mathfrak{L}^n_i: (\mathbb{O}^n)^{(\op,\op)} \to \msSet 
	\]
	given on objects by 
	\[
	j\mapsto \left(\bigcup_{\substack{I\subset [j,n]\\\Delta^I\subset \Lambda^n_i}} \OO^{[j,n]}(j,\min(I))^{\op}\times \D^I\right)^\flat\subset \left(\D^{[j,n]}\right)^\flat
	\]
	with enriched functoriality given by the formulae 
	\[
	\OO^n(k,j)^\op \times \OO^{[j,n]}(j,\min(I))^\op\times \D^I\to \OO^{[k,n]}(k,\min(I))^\op\times \D^I; \quad (S,T,U)\mapsto (S\cup T, U).
	\]
\end{definition}

To ease writing, we refer to those subset $I\subset [n]$ such that $\Delta^I\subset \Lambda^n_i$ as \emph{$i$-admissible}. 

\begin{lemma}\label{lem:adjtoLambd}
	Consider $(\Delta^n)^\flat \to\Delta^n$ and $(\Lambda^n_i)^\flat\to \Delta^n$ as objects in $(\msSet)_{/\Nsc(\OO^n)}$.
	\begin{enumerate}
		\item For $n\geq 0$, 
		\[
		\PPhi_{\OO^n}((\Delta^n)^\flat)\cong \mathfrak{D}^n.
		\]
		\item\label{adjtoLamb:2} For $0<i<n$, 
		\[
		\PPhi_{\OO^n}((\Lambda^n_i)^\flat)\cong \mathfrak{L}^n_i.
		\]
	\end{enumerate}
\end{lemma}

\begin{proof}
	Follows from unraveling the definitions.
\end{proof}

It is work noting that for $i,j>0$, $[j,n]\subset [n]$ is itself an $i$-admissible set. Consequently, we have that, for $j>0$, $\mathfrak{L}^n_i(j)\cong \mathfrak{D}^n(j)$. This reduces the problem of showing that $\mathfrak{L}^n_i\Rightarrow \mathfrak{D}^n$ is an objectwise marked weak equivalence to the problem of showing that $\mathfrak{L}^n_i(0)\to \mathfrak{D}^n(0)$ is a marked weak equivalence. Moreover, since both marked simplicial sets carry the minimal marking, it will suffice to show that this is a weak equivalence in the Joyal model structure.  

\begin{definition}
	For ease of notation, we set
	\[
	\L^n_i:= \mathfrak{L}^n_i(0) = \bigcup_{\Delta^J\subset \Lambda^n_i} \A(J) \subset \Dn.
	\]
\end{definition}

We now turn to showing that the map of simplicial sets
\[
\Lni \hra \Dn
\]
is a trivial cofibration for the Joyal model structure. This in turn is equivalent to the statement
that the induced functor of simplicial categories 
\[
\Cf[\Lni] \hra \Cf[\Dn]
\] 
is a weak equivalence. This latter functor induces a bijection on objects. Furthermore, since $\Dn$ is the
nerve of a poset, the simplicial set $\Cf[\Dn](S,T)$ is contractible if $S \le T$ and empty otherwise.
Therefore, Proposition \ref{prop:innhorn} is an immediate corollary of the following main result of this section:

\begin{theorem}\label{thm:contractible}
	Let $0 < i <n$ and let $S \le T$ be elements of $D^n$. Then the simplicial set
	$\Cf[\Lni](S,T)$ is contractible.
\end{theorem}

\begin{corollary}\label{cor:trivjoyal} For every $0 < i < n$, the map of simplicial sets
	\[
	\Lni \hra \Dn
	\]
	is a trivial cofibration for the Joyal model structure.
\end{corollary}

To prepare the proof of Theorem \ref{thm:contractible}, we begin with an explicit description of the
mapping spaces $\Cf[\K](S,T)$ where $\K \subset \D^n$ is a simplicial subset. We define a \emph{strict chain of length $\ell$ between $S$ and $T$} to be a sequence 
\[
S = M_0 < M_1 < M_2 < \cdots < M_{\ell-1} < M_{\ell} = T 
\]
of objects in $D^n$. Let $P^n(S,T)$ denote the poset of strict chains
between $S$ and $T$ in $D^n$ with length $\ell > 0$, ordered by refinement $\underline{M} \subset \underline{M}^\prime$. Furthermore,
let $\P^n(S,T)$ denote the nerve of $P^n(S,T)$. \glsadd{Pn} We define
\[
\P^n_{\K}(S,T) \subset \P^n(S,T)
\]
to be the simplicial subset consisting of those simplices whose corresponding sequence of chain
refinements
\[
\underline{M}^{(0)} \subset \underline{M}^{(1)} \subset \cdots \subset \underline{M}^{(k)}
\]
satisfies the following condition:
\begin{itemize}
	\item for every consecutive pair of elements $M_r^{(0)} < M_{r+1}^{(0)}$ in the chain
	$\underline{M}^{(0)}$, the simplex of $\Dn$ corresponding to the totally ordered
	chain
	\[
	\Big\{M^\prime \in \underline{M}^{(k)} \;\Big|\; M_r^{(0)} < M^\prime < M_{r+1}^{(0)} \Big\} \subset D^n
	\]
	is contained in $\K$.
\end{itemize}
With this terminology, we have the following:

\begin{proposition}\label{prop:mapping} Let $\K \subset \D^n$ be a simplicial subset containing the
	pair $(S,T)$ of vertices of $\D^n$. Then there is an isomorphism of simplicial sets
	\[
	\Cf[\K](S,T) \cong \P^n_{\K}(S,T).
	\]
\end{proposition}
\begin{proof}
	This follows by explicit computation of $\Cf[\K](S,T)$, representing $\K$ as a colimit over
	its nondegenerate simplices based on \cite[1.1.5.9]{LurieHTT}. It can also be seen very
	nicely, by using the Dugger-Spivak necklace model for ordered simplicial sets
	\cite[4.11]{Dugger-Spivak}. 
\end{proof}

For the remainder of this section, we will always identify mapping spaces of the form $\Cf[\K](S,T)$
with the model $\P^n_{\K}(S,T)$ provided by Proposition \ref{prop:mapping}. In particular, we will
implicitly use this chain model to describe the mapping spaces of the simplicial category
$\Cf[\Lni]$.

We will now introduce some terminology which will allow us to get a handle on the combinatorics of
the poset $D^n$, the simplicial set $\Lni$, and its associated simplicial category $\Cf[\Lni]$. For the
remainder of the section, we fix $n \ge 2$, and $0 < i < n$. We first note that the poset $D^n$ can
be realized as a full subposet of $\ZZ^{n-1}$. Namely, let $e_j$ denote the $j$th unit coordinate
vector of $\ZZ^{n-1}$ and consider the embedding\glsadd{x}
\[
x: D^n \hra \ZZ^{n-1},\; J \mapsto \max(J) e_{n-1} + \sum_{\substack{0 < j < \max(J)\\j \notin
		J}} e_j.
\]
In fact, the visualizations of the posets in Figure \ref{fig:posets} are obtained precisely via this embedding.
We further denote by $x_i$ the postcomposition of $x$ with projection to the $i$th coordinate of
$\ZZ^{n-1}$. Various useful notions arise from these geometric coordinates on $D^n$: 

\begin{definition} Let $S \le T$ be a pair of comparable elements of $D^n$.
	\begin{enumerate}
		\item The edge $S \le T$ is called {\em atomic} if $x(T) - x(S)$ is a unit coordinate
		vector of $\ZZ^{n-1}$. Note that being atomic means that $S < T$ and there does not
		exist $M \in D^n$ with $S < M < T$. 
		\item We define the {\em atomic distance} between $S$ and $T$ as the taxicab distance
		\[
		d(S,T) := \sum_{1 \le i \le n-1} x_i(T) - x_i(S).
		\]
		Note that the atomic distance between $S$ and $T$ measures the length of any chain
		\[
		S = M_0 < M_1 < \cdots < M_{l} = T
		\]
		with $M_r < M_{r+1}$ atomic. 
		\item We call $S$ and $T$ {\em close} if there exists a rectilinear unit cube in
		$\ZZ^{n-1}$ containing the points $x(S)$ and $x(T)$ among its vertices.
		Equivalently, $S$ and $T$ are close if $x_{n-1}(T) - x_{n-1}(S) \le 1$.
		\item We call $S$ and $T$ {\em distant} if they are not close.\\
	\end{enumerate}
\end{definition}

It will be convenient to introduce notation for a certain class of $i$-admissible subsets of $[n]$.
A subset $J \subset [n]$ is called {\em $i$-superior} if it is maximal among all $i$-admissible
subsets with the same minimal element. Note that, by maximality, the collection of simplicial sets
$\{\A(J)\}$ where $J$ runs through all $i$-superior subsets covers $\Lni$:\glsadd{Lni}
\[
\Lni = \bigcup_{\Delta^J\subset\Lambda^n_i} \A(J) =
\bigcup_{\substack{\Delta^J\subset\Lambda^n_i\\\text{$i$-superior}}} \A(J)\subset \Dn.
\]
This covering indexed by $i$-superior sets will play a central role in the proof of Theorem
\ref{thm:contractible}.

\begin{remark}\label{rem:superior} 
	The $i$-superior subsets of $[n]$ are precisely the subsets of the form
	\begin{enumerate}
		\item $[n] \setminus \{j\}$ where $j \ne i$,
		\item $\{k,k+1,\cdots,n\}$ where $k \ge 1$.
	\end{enumerate}
\end{remark}

\begin{proof}[Proof of Theorem \ref{thm:contractible}]
	
	The proof will be given by induction on the atomic distance between $S$ and $T$. 
	
	For the base of the induction, suppose that $S$ and $T$ have atomic distance $1$. Then they are
	close so that, by Lemma \ref{lem:close}, the edge $S < T$ is contained in $\Lni$. Therefore,
	$\Cf[\Lni](S,T)$ consists of a single vertex representing this edge and is hence contractible.
	
	Suppose now, for the inductive step, that $S$ and $T$ have atomic distance $d \ge 2$ and assume that, for every pair $S' <
	T'$ with atomic distance $d(S',T') < d$, the simplicial set $\Cf[\Lni](S',T')$ is contractible. We
	distinguish three cases:
	\begin{enumerate}[label={(\Roman*)}]
		\item\label{c:1}{\em There does not exist an $i$-superior subset $J$ such that $A(J)$
			contains both $S$ and $T$.}
		
		This implies that every vertex of $\Cf[\Lni](S,T)$ corresponds to a chain
		between $S$ and $T$ of length $\ge 2$. Further, by Lemma \ref{lem:close}, $S$ and $T$ must be distant
		so that, by Lemma \ref{lem:distant}, the poset $P^n(S,T)_{\ge 2} \subset P^n(S,T)$ of chains between
		$S$ and $T$ of length $\ge 2$ is contractible.
		The functor
		\[
		\begin{aligned}
		p:\; (P^n(S,T)_{\ge 2})^{\op} &\lra \sSet,\\
		\big\{S < M_1 < \cdots M_k < T \big\} &\mapsto \Cf[\Lni](S,M_1) \times
		\Cf[\Lni](M_1,M_2) \times \cdots \times \Cf[\Lni](M_k,T)
		\end{aligned}
		\]
		satisfies 
		\[
		\colim p \cong \Cf[\Lni](S,T).
		\]
		By Lemma \ref{lem:cofibrant}, the diagram $p$ is projectively cofibrant, so that, by Remark \ref{rem:cofibrant}, the
		latter colimit is a homotopy colimit. By induction hypothesis, the values of $p$ are contractible
		simplicial sets so that, since $P^n(S,T)_{\ge 2}$ is contractible as well, we deduce that
		$\Cf[\Lni](S,T)$ is contractible.
		
		\item\label{c:2}{\em There exists an $i$-superior subset $J$ such that $A(J)$ contains every $M
			\in D^n$ with $S \le M \le T$.} 
		
		By Proposition \ref{prop:mapping}, we have that
		\[
		\Cf[\Lni](S,T) \cong \P^n(S,T)
		\]
		where the latter simplicial set is the nerve of the poset $P^n(S,T)$ with minimal element $S
		< T$ and hence contractible.
		
		\item\label{c:3}{\em There exists an $i$-superior subset $J$ such that $A(J)$ contains
			$S$ and $T$, but there does not exist an $i$-superior subset satisfying the
			hypothesis of \ref{c:2}.}
		
		This is the most subtle case, since both chains between $S$ and $T$ of length
		$\ge 2$ which may not be contained completely in any $A(J)$ as well as chains
		between $S$ and $T$ of length $1$ which are contained in some $A(J)$ contribute to
		$\Cf[\Lni](S,T)$. In a sense, the argument in this case will be a hybrid of the
		arguments for the extreme cases \ref{c:1} and \ref{c:2}.
		
		Denote by $\U_i(S,T)$ the poset whose elements are (possibly empty) sets
		\[
		\underline{J} = \{J_1, \cdots, J_k\}
		\]
		where each $J_r$ is an $i$-superior subset of $[n]$ such that $A(J)$ contains $S$
		and $T$. Define the poset
		\[
		\V_i(S,T) := (\U_i(S,T) \times \{1 < 2\}) \setminus \{(\emptyset, 1)\}
		\]
		and further a functor
		\[
		p: (\V_i(S,T))^{\op} \lra \Cat,\quad  (\underline{J},m) \mapsto \begin{cases}
		P^n_{A(\underline{J})}(S,T) & \text{for $m=1$,}\\
		P^n_{A(\underline{J})}(S,T)_{\ge 2} & \text{for $m=2$,} \end{cases}
		\]
		where 
		\[
		P^n_{A(\underline{J})}(S,T) \subset P^n(S,T)
		\]
		denotes the subposet consisting of those chains between $S$ and $T$ which are
		contained in every poset $A(J)$, $J \in \underline{J}$. Denote by $\chi$ the
		Grothendieck construction of $p$, as defined in the introduction. The category
		$\chi$ is again a poset, an element of $\chi$ is given by a triple
		\[
		(\underline{J},l, \{S < M_1 < \cdots < M_k < T\})
		\]
		consisting of a set $\underline{J} \in \U_i(S,T)$ of $i$-superior subsets, a number
		$l \in \{1,2\}$, and a chain between $S$ and $T$ of length $\ge l$ lying in all
		posets $A(J)$, $J \in \underline{J}$.  
		
		The poset $\V_i(S,T)$ has a maximal element $(\underline{J},2)$ where
		$\underline{J}$ is the set of all $i$-superior subsets $J$ such that $A(J)$ contains
		$S$ and $T$. Hence it is contractible.
		By Lemma \ref{lem:distant} and Lemma \ref{lem:admissible}, for every $\underline{J}
		\in \U_i(S,T)$, the poset $P^n_{A(\underline{J})}(S,T)_{\ge 2}$ is contractible. The
		posets $P^n_{A(\underline{J})}(S,T)$ are contractible since they have a minimal
		element given by the chain $S < T$ of length $1$.
		Therefore, by Lemma \ref{lem:colimit}, the poset $\chi$ is contractible.
		
		To conclude, define a functor $q: \chi^{\op} \to \sSet$ by assigning to the element 
		\[
		(\underline{J}, l, \{S < M_1 < \cdots < M_k < T\})
		\]
		the simplicial set
		\[
		\Cf[\A(\underline{J})](S,M_1) \times \cdots \times \Cf[\A(\underline{J})](M_k,T)
		\]
		where 
		\[
		\A(\underline{J}) = \begin{cases} 
		\Lni & \text{if $\underline{J} = \emptyset$,}\\
		\cap_{J \in \underline{J}} \A(J) & \text{else.}
		\end{cases}
		\]
		Note that the values of $q$ are contractible simplicial sets: for $\underline{J} \ne
		\emptyset$ this follows, since $\A(\underline{J})$ is the nerve of the poset
		$\cap_{J \in \underline{J}} A(J)$; for $\underline{J} = \emptyset$, we have $l = 2$,
		so that the contractibility of the involved mapping spaces follows from the
		induction hypothesis.
		By Remark \ref{rem:cofibrant}, the diagram $q$ is projectively cofibrant and we have, by construction,
		\[
		\colim q \cong \Cf[\Lni](S,T).
		\]
		Therefore, by Lemma \ref{lem:cofibrant}, this latter colimit is a homotopy colimit of contractible simplicial sets
		parametrized by the contractible poset $\chi$. Consequently, $\Cf[\Lni](S,T)$ is
		contractible. \qedhere
	\end{enumerate}
\end{proof}

\begin{lemma}\label{lem:close} Suppose that $S \le T$ in $D^n$ are close. Then there exists
	an $i$-admissible subset $J \subset [n]$ such that $A(J)$ contains every $M \in D^n$ with $S
	\le M \le T$. 
\end{lemma}
\begin{proof} 
	This is geometrically clear from the embedding $x: D^n \hra \ZZ^{n-1}$: Every rectilinear
	unit $k$-cube in $x(D^n)$, where $1 \le k \le n-1$, is contained in one of the following
	subsets
	\begin{enumerate}
		\item $x(A([n] \setminus \{n\}))$,
		\item $x(A(\{n-1,n\}))$,
		\item $x(A(\{n\}))$.
	\end{enumerate}
	If $x(S)$ and $x(T)$ are vertices of the same rectilinear unit cube, then the same must be true
	for $x(M)$ where $S < M < T$. Therefore, we may take $J$ to be one of the subsets $[n]
	\setminus \{n\}$, $\{n-1,n\}$, and $\{n\}$, respectively.
\end{proof}

\begin{lemma}\label{lem:distant}
	Let $S \le T$ be distant elements of $D^n$. Then the subposet
	\[
	P^n(S,T)_{\ge 2} \subset P^n(S,T),
	\]
	consisting of chains of length $\ge 2$, is contractible.
\end{lemma}

\begin{proof}
	Let $G(S,T)\subset D^n$ denote the full subposet on those $V$ such that $S<V<T$. By definition, $P^n(S,T)_{\ge 2}$ is the category of non-degenerate simplices of $G(S,T)$, and thus has the same homotopy type as $G(S,T)$. It will therefore suffice to show that $G(S,T)$ is contractible. 
	
	Denote by $G(S,T)^>$ the full subposet of $G(S,T)$ on those objects $V$ for which $x_{n-1}(V)>x_{n-1}(S)$. We then make two observations: 
	\begin{enumerate}
		\item $G(S,T)^>$ has an initial element, given by the preimage of $x(S)+e_{n-1}$, and is thus contractible.  
		\item The morphism $r:G(S,T)\to G(S,T)^>$ which maps 
		\[
		V\mapsto \begin{cases}
		x(V)+e_{n-1} & x_{n-1}(V)=x_{n-1}(S)\\
		x(V) & \text{else}
		\end{cases}
		\]
		is well-defined (since $S$ and $T$ are distant) and defines a homotopy inverse to the inclusion $G(S,T)^>\hookrightarrow G(S,T)$. Consequently, $G(S,T)$ is contractible. \qedhere
	\end{enumerate} 
\end{proof}

\begin{lemma}\label{lem:admissible}
	Let $S \le T$ elements of $D^n$ satisfying the hypothesis of \ref{c:3} in the proof of
	Theorem \ref{thm:contractible}. Let $\underline{J} \in \U_i(S,T)$ be a set of $i$-superior
	subsets. Then the subposet 
	\[
	P^n_{A(\underline{J})}(S,T)_{\ge 2} \subset P^n(S,T)_{\ge 2},
	\]
	consisting of those chains which are contained in every $A(J)$, $J \in \underline{J}$, is contractible.
\end{lemma}
\begin{proof}
	First note, that the poset $P^n_{A(\underline{J})}(S,T)_{\ge 2}$ can be identified
	with the category of simplices of the poset $G(S,T)_{\underline{J}}$ consisting of
	those objects $V$ such that $S < V < T$ and such that, for every $J \in
	\underline{J}$, we have $V \in J$. In particular, it suffices to show that the
	posets $G(S,T)_{\underline{J}}$ are contractible.
	
	The $i$-superior subsets $J \in \underline{J}$ must be of the form
	$[n]\setminus{j}$ with $j > \max(S)$ and $j \neq n,i$. Indeed, all other $i$-superior sets
	would contradict the hypothesis of \ref{c:3}. In particular,
	this implies that $\max(T)=n$ and that $S \cap T=\{0\}$. To simplify notation, we
	will identify $\underline{J}=\{[n]\setminus{j_\ell}\}_{\ell=1}^{r}$ with the set
	$X:=\{j_1,j_2,\dots,j_r\}$ and distinguish the following two cases:
	
	\begin{enumerate}
		\item Suppose $S = \{ 0 \}$. Let $k \in [n]\setminus \{0\}$ be the smallest
		element such that $k \notin X$. We observe that $ \{0,k\} \in
		G(S,T)_{\underline{J}}$ and define $G(S,T)_{\underline{J}}^{\leq k}$
		to be the full subposet on those objects $V$ such that $V \leq
		\{0,k \}$. The inclusion functor
		\[
		i_k: G(S,T)_{\underline{J}}^{\leq k} \subset G(S,T)_{\underline{J}},
		\]
		admits a section
		\[
		r_k: G(S,T)_{\underline{J}} \lra G(S,T)_{\underline{J}}^{\leq k},\; V \mapsto \begin{cases}
		V, & \text{if $\max(V) \leq k$,} \\
		(V \cap [k]) \cup \{k\} & \text{otherwise.}
		\end{cases}
		\]
		Furthermore, we have $i_k \circ r_k(V) \leq V$ thus inducing a
		natural transformation between $i_k \circ r_k$ and $\id$. This shows
		that both posets have the same homotopy type. Since
		$G(S,T)_{\underline{J}}^{\leq k}$ has a final element $\{0,k\}$, we see
		that $|G(S,T)_{\underline{J}}|$ is contractible.
		
		\item Suppose $S \neq \{0\}$ and set $s = \max(S)$. We define
		$G(S,T)_{\underline{J}}^{s} \subset G(S,T)_{\underline{J}}$
		consisting of those sets $V$ which contain $s$.
		This poset has a final element given by $T \cup \{ s \}$. As
		above, the inclusion $i_s: G(S,T)_{\underline{J}}^{s} \subset
		G(S,T)_{\underline{J}}$ admits a section
		\[
		r_s: G(S,T)_{\underline{J}} \lra
		G(S,T)^{s}_{\underline{J}},\; V \mapsto \begin{cases}
		V & \text{if $s \in V$,} \\
		V \cup \{s\} & \text{otherwise.}
		\end{cases}
		\]
		which is further a homotopy inverse. As above, this implies the
		contractibility of $|G(S,T)_{\underline{J}}|$ concluding the proof.\qedhere
	\end{enumerate}
\end{proof}

For $\C$ a small category and $F:\C\to \sSet$ a diagram, we call $F$ \emph{projectively
	cofibrant} if the natural transformation $\emptyset\to F$ from the constant diagram to $F$
has the left lifting property with respect to all pointwise trivial fibrations in the
Kan-Quillen model structure. 

\begin{remark}\label{rem:cofibrant}
	The projectively cofibrant functors are of use to us precisely because, for a projectively cofibrant functor $F:\C\to \sSet$, the canonical morphism 
	\[
	\hocolim_{\C} F \to \colim_{\C} F
	\]
	is a weak equivalence. cf. \cite[A.2.8]{LurieHTT}. 
\end{remark}

\begin{lemma}\label{lem:cofibrant}
	Let $X$ be a simplicial set, $\{U_i\}_{i\in I}$ a finite cover of $X$ by non-empty
	simplicial subsets, and $P\subset \P(I)\setminus \emptyset$ a full sub-poset
	containing all $J \subset I$ such that $\bigcap_{j\in J}U_j\subset X$ is non-empty.
	Then the diagram
	\[
	F:P^{\op} \lra \Set_\Delta,\; J \mapsto \bigcap_{j\in J} U_j
	\]
	is projectively cofibrant. 
\end{lemma}

\begin{proof}
	Since the source of $F$ is a finite poset and every simplicial set is cofibrant, an argument
	similar to that of \cite[Prop. A.2.9.19]{LurieHTT} shows that it is sufficient to check that
	$\colim_{P^{\op}_{/J} \setminus \{J\}} F \to F(J)$ is a cofibration. A quick calculation shows that 
	\[
		\colim_{P^{\op}_{/J} \setminus \{J\}} F \cong \bigcup_{K\supsetneq J} \left(\bigcap_{k\in K} U_k\right)\subset F(J),
	\]
	completing the proof. 
\end{proof}

\begin{lemma}\label{lem:colimit} 
	Let $P$ be a contractible poset, and let $F:P^{\op}\to \Cat$ be a diagram such that $F(p)$ is contractible for all $p\in P$. Then the Grothendieck construction $\chi(F)$ is contractible.  
\end{lemma}

\begin{proof} This follows from the well-known fact that the geometric realization of the Grothendieck
	construction computes the homotopy colimit of the geometric realizations of the various categories of
	the diagrams (as can, for example, be deduced from \cite[3.3.4.3]{LurieHTT}). A direct
	argument adapted to the given situation is the following:
	We consider the resulting Cartesian fibration $\pi:\chi(F)\to P.$ For $p\in P$, we
	relate the fiber $F(p)$ to the overcategory $\chi(F)_{p/}$. Clearly there is an inclusion 
	\[
	i:F(p) \lra \chi(F)_{p/},\; c \mapsto (p,c)
	\]
	Setting $f_q$ to be the unique morphism in $P$ from $p$ to $q$, we can also define a functor 
	\[
	r: \chi(F)_{p/} \lra F(p),\; (q,c) \mapsto F(f_q)(c)
	\]
	It is immediate that $r\circ i=\id_{F(p)}$. 
	
	We then note that the canonical morphisms $(p,F(f_q)(c))\to (q,c)$ given by
	$(f_q,\id_{F(f_q)(c)})$ form a natural transformation $i\circ r\to
	\id_{\chi(F)_{p/}}$. Consequently, $i$ is a homotopy equivalence, and so by
	assumption the overcategories $\chi(F)_{p/}$ are all contractible. By Quillen's
	Theorem A, $\pi$ induces a homotopy equivalence of nerves, and so since $P$ is
	contractible, $\chi(F)$ is contractible.  
\end{proof}	

\section{The Quillen equivalence}\label{sec:qe}

The goal of this section will be to show that the adjunction 
\[
	\PPhi_{\CC}: (\msSet)_{/\Nsc(\CC)} \llra \Fun_{\msSet}(\CC^{(\op,\op)}, \msSet): \cchi_{\CC},
\]
from \eqref{eq:adjunction} is in fact a Quillen equivalence. The proof strategy precisely parallels
the argument of \cite{LurieHTT} to show that the ordinary relative nerve defines a Quillen
equivalence: we first relate the values of $\PPhi_{\CC}$ and the contravariant scaled straightening
functor $\St_{\CC}$ from Definition \ref{defn:CSt} on simplices over $\Nsc(\CC)$. We then show that
the comparison maps thus obtained glue to produce global natural comparison maps between the values
of $\PPhi_{\CC}$ and $\St_{\CC}$ on any marked simplicial set over $\Nsc(\CC)$. We show that this
natural transformation is an objectwise weak equivalence. Since $\PPhi_{\CC}$ preserves cofibrations
this comparison shows, by 2-out-of-3, that it also preserves trivial cofibrations and is thus a
Quillen adjunction. In addition, this comparison with Lurie's straightening functor shows that the
left derived functors of $\PPhi_{\CC}$ and $\St_{\CC}$ are equivalent, allowing us to conclude the
proof.

\subsection{Base change}

Given a $2$-category $\CC$, the contravariant scaled straightening functor $\St_{\CC}$ associates to a
marked simplicial set $X \in \msSet_{/\Nsc(\CC)}$ the $\msSet$-enriched functor
\[
\St_{\CC}(X): \CCop \lra \msSet
\]
whose value at $c \in \CC$ is the marked simplicial set
\[
\St_{\CC}(X)(c) = \Map_{\CC(X)}(c,v)^{op}
\]
of maps in the $\msSet$-enriched category
\begin{equation}\label{eq:simcat}
\CC(X) = \CC \coprod_{\Cfsc[\Nsc(\CC)]} \Cfsc\Big[\Nsc(\CC) \coprod_{(X \times \{0\})_{\flat}} X \times \Delta^1 \coprod_{(X
	\times \{1\})_{\flat}} \{*\}\Big].
\end{equation}

To relate the straightening functor $\St_{\CC}$ from Definition \ref{defn:CSt} to the left adjoint $\PPhi_{\CC}$ we begin by analyzing its base change
behaviour with respect to $2$-functors $\CC \to \DD$. 

\begin{proposition}
	Let $f: \CC \to \DD$ be a functor between 2-categories. Then the diagram
	\begin{equation}\label{eq:basechange}
	\begin{tikzcd}
	\Fun_{\msSet}(\DD^{(\op,\op)},\msSet)  \ar{r}{f^*}\ar{d}{\cchi_{\DD}} &
	\Fun_{\msSet}(\CC^{(\op,\op)},\msSet)\ar{d}{\cchi_{\CC}}\\
	(\msSet)_{/\Nsc(\DD)}  \ar{r}{\Nsc(f)^*} & (\msSet)_{/\Nsc(\CC)},
	\end{tikzcd}
	\end{equation}
	with horizontal morphisms are given by pullback along $f$ and $\Nsc(f)$, respectively,
	commutes up to natural isomorphism. 
\end{proposition}
\begin{proof}
	This is immediate from the definition of the relative $2$-nerve.
\end{proof}

\begin{corollary}\label{cor:basechange}
	Let $f: \CC \to \DD$ be a functor between 2-categories. Then the diagram
	\[
	\begin{tikzcd}
	\Fun_{\msSet}(\DD^{(\op,\op)},\msSet) & \ar[swap]{l}{f_!}
	\Fun_{\msSet}(\CC^{(\op,\op)},\msSet)\\
	(\msSet)_{/\Nsc(\DD)} \ar{u}{\PPhi_{\DD}}   &
	\ar[swap]{l}{\Nsc(f)_!}(\msSet)_{/\Nsc(\CC)}\ar{u}{\PPhi_{\CC}} ,
	\end{tikzcd}
	\]
	obtained from \eqref{eq:basechange} by passing to left adjoints of the horizontal functors,
	commutes up to natural isomorphism. 
\end{corollary}

\subsection{Comparison for simplices}
\label{subsec:simplices}

We now provide a map\glsadd{etaC}
\begin{equation}\label{eq:natural}
\eta_{\CC}(X): \St_{\CC}(X) \lra \PPhi_{\CC}(X)
\end{equation}
for the following choices of $\CC$ and $X$:
\begin{enumerate}
	\item $\CC = \OO^n$ and $X = (\Delta^n)^{\flat}$,
	\item $\CC = \OO^1$ and $X = (\Delta^1)^{\sharp}$.
\end{enumerate}
Following \cite[3.2.5.10]{LurieHTT}, it will then be shown that these choices canonically extend to
determine maps $\eta_{\CC}(X): \St_{\CC}(X) \lra \PPhi_{\CC}(X)$, for all $\CC$ and $X \in
(\msSet)_{/\Nsc(\CC)}$.

\begin{enumerate}
	\item $\CC = \OO^n$ and $X = (\Delta^n)^{\flat}$. Per Lemma \ref{lem:adjtoLambd}, we have that $\PPhi_{\OO^n}((\Delta^n)^{\flat})\cong \mathfrak{D}^n$. 
	To determine the value of the straightening functor, note that formula \eqref{eq:simcat} yields the
	$\msSet$-enriched category 
	\[
	\OO^n( (\Delta^n)^{\flat}) = \OO^n \coprod_{\Cfsc[\Delta^n_{\flat} \times \{0\}]} \Cfsc[\Delta^n_{\flat}
	\times \Delta^1] \coprod_{\Cfsc[\Delta^n_{\flat} \times \{1\}]} \Cfsc[\{*\}].
	\]
	For $i,j \in \OO^n$, the marked simplicial set
	\[
	\Map_{\Cfsc[\Delta^n_{\flat} \times \Delta^1]}((i,0),(j,1))
	\]
	can be identified with the nerve of the poset $P(i,j)$ of chains in the poset $[n] \times [1]$ between
	$(i,0)$ and $(j,1)$, ordered by refinement. We define the map
	\[
		P(i,j)^{\op} \lra D^{[i,n]}, C \mapsto C_0 \cup \min(C_1)
	\]
	where we set $C_i := C \cap ([n] \times \{i\})$. Passing to nerves, one verifies that this map
	factors to define a map 
	\begin{equation}\label{eq:map}
	\Map_{\OO^n( (\Delta^n)^{\flat})}(i,*)^{\op} \lra \D^{[i,n]}
	\end{equation}
	natural in $i$. This yields the desired map
	\begin{equation}\label{eq:nsimplex}
	\St_{\OO^n}(\Delta^n_{\flat}) \lra \PPhi_{\OO^n}(\Delta^n_{\flat}).
	\end{equation}
	
	\item $\CC = \OO^1$ and $X = (\Delta^1)^{\sharp}$. On underlying unmarked simplicial sets,
	we may use the maps \eqref{eq:map} to define the desired map
	\begin{equation}\label{eq:1simplex}
	\St_{\OO^1}( (\Delta^1)^{\sharp}) \lra \PPhi_{\OO^1}( (\Delta^1)^{\sharp}).
	\end{equation}
	We then conclude by observing that the markings are compatible as well.
\end{enumerate}

\subsection{Comparison for simplicial sets}

We now extend the comparison maps \eqref{eq:natural} from simplices to simplicial sets following a
standard technique from \cite{LurieHTT}.

\begin{proposition}\label{prop:natural} There exists a unique family of natural transformations
	\begin{equation}\label{eq:naturaltrafo}
	\eta_{\CC}(X): \St_{\CC}(X) \lra \PPhi_{\CC}(X)
	\end{equation}
	indexed by pairs $(\CC, X)$ where $\CC$ is a $2$-category and $X \in (\msSet)_{/\Nsc(\CC)}$
	with the following properties:
	\begin{enumerate}[label=(\arabic*)]
		\item For every map $g: X \to Y$ of marked simplicial sets over $\Nsc(\CC)$, the
		diagram
		\[
		\begin{tikzcd}
		\St_{\CC}(X) \ar[swap]{d}{\St_{\CC}(g)}\ar{r}{\eta_{\CC}(X)} &
		\PPhi_{\CC}(X) \ar{d}{\PPhi_{\CC}(g)}\\
		\St_{\CC}(Y) \ar{r}{\eta_{\CC}(Y)} & \PPhi_{\CC}(Y)\\
		\end{tikzcd}
		\]
		commutes. 
		\item For every $2$-functor $f: \CC \to \DD$, the diagram
		\[
			\begin{tikzcd}[column sep=2cm]
		f_! \St_{\CC}(X) \ar[swap]{d}{\cong} \ar{r}{f_! \circ \eta_{\CC}(X)} &
		f_! \PPhi_{\CC}(X) \ar{d}{\cong}\\
		\St_{\DD}(\Nsc(f)_! X) \ar{r}{\eta_{\CC}(\Nsc(f)_! X)} & \PPhi_{\DD}(\Nsc(f)_! X)\\
		\end{tikzcd}
		\]
		commutes. 
		\item For $\CC = \OO^n$ and $X = (\Delta^n)^{\flat}$, the map $\eta_{\CC}(X)$
		coincides with the map \eqref{eq:nsimplex}.
		\item For $\CC = \OO^1$ and $X = (\Delta^1)^{\sharp}$, the map $\eta_{\CC}(X)$
		coincides with the map \eqref{eq:1simplex}.
	\end{enumerate}
\end{proposition}

The proof of Proposition \ref{prop:natural} is a routine application of arguments like those
of \cite[Rem 3.2.5.10]{LurieHTT}. One first uses base change to compute the value of
$\eta_{\OO^n}$ on any simplex of $\Nsc(\OO^n)$, and then shows naturality on the full
subcategory on those objects. The remainder of the argument is completely parallel to loc.
cit.

\begin{proposition}\label{prop:objectwisewe}
	For every $2$-category $\CC$ and every $X \in (\msSet)_{/\Nsc(\CC)}$, the map
	\[
		\eta_{\CC}(X): \St_{\CC}(X) \lra \PPhi_{\CC}(X)
	\]
	from \eqref{eq:naturaltrafo} is a weak equivalence with respect to the projective model
	structure. 
\end{proposition}
\begin{proof}
	We first show that $\eta_{\CC}$ is a weak equivalence on objects of the
	form 
	\[
	(\Delta^n)^{\flat} \to \Nsc(\CC), \quad\text{$n \le 1$,}
	\]
	and
	\[
	(\Delta^n)^{\sharp} \to \Nsc(\CC), \quad\text{$n \le 1$.}
	\]
	The case $n=0$ is trivial. For $n=1$, since left Kan extension preserves weak
	equivalences between cofibrant objects, we reduce to the case of the identity map
	$\id: (\Delta^1)^{\flat} \to \Nsc(\OO^1)$. Then we have
	\[
	\St_{\OO^1}((\Delta^1)^{\flat})(1) \cong
	\PPhi_{\OO^1}((\Delta^1)^{\flat})(1) \cong \Delta^0.
	\] 
	Similarly, one observes that,
	\[
	\St_{\OO^1}((\Delta^1)^{\flat})(0) \cong (\Lambda^2_2)^{\dag} 
	\]
	where $(\Lambda^2_2)^{\dag}$ denotes the right horn with degenerate edges and $\{1,2\}$
	marked, and the map to $\D^{[0,1]}$ is given by collapsing the marked edge, hence a marked
	equivalence. 
	The case of $(\Delta^1)^{\sharp} \to \Nsc(\OO^1)$ follows analogously,
	keeping track of the new marked edges.
	The claim now follows from Proposition \ref{prop:innhorn} and Lemma \ref{lem:eq} below.
\end{proof}

\begin{lemma}\label{lem:eq}
	Let $\overline{K}=(K,K^{\th})$ be a scaled simplicial set. Suppose we are given two left adjoint functors, 
	\[
	 	L_1,L_2: (\msSet)_{/\overline{K}} \to C,
	\]
	where $C$ is a left proper combinatorial model category and $L_1$ is a left Quillen functor.
	Suppose further that $L_2$ preserves cofibrations and that it maps the morphisms,
	\begin{equation}\label{eq:hornfillcond}
		(\Lambda^n_i)^{\flat} \to (\Delta^n)^{\flat}, \enspace n \geq 2, \enspace 0<i<n,
	\end{equation}
	to weak equivalences. Given a natural transformation
	$\eta: L_1 \Rightarrow L_2$ which is a weak equivalence on objects of the form
	\[
	(\Delta^n)^{\flat} \to K \quad \text{$n \leq 1$,}
	\]
	and	
	\[
	(\Delta^n)^{\sharp} \to K, \quad \text{$n \leq 1$.}
	\]
	Then $\eta$ is a levelwise weak equivalence.
\end{lemma}
\begin{proof}
	Every marked simplicial set over $X$ can be expressed as a filtered colimit of its skeleta.
	Since $C$ is a combinatorial model category, weak equivalences are stable under
	filtered colimits. This shows that we can reduce to the case of objects,
	\[
	X \lra K, \quad\text{$X$ finite dimensional.}
	\]
	In addition we know that every simplicial set can be expressed as a filtered colimit of
	simplicial sets containing only finitely many non-degenerate simplices, allowing us to
	further reduce to the latter case. Recall that given a pushout diagram,
	\[
	\begin{tikzcd}
	X_1 \arrow[r] \arrow[d] & X_2 \arrow[d] \\
	Y_1 \arrow[r] & Y_2
	\end{tikzcd}
	\]
	in $(\msSet)_{/ \overline{K}}$ with one of the morphisms $X_1 \to Y_1$ or $X_1\to X_2$ a
	cofibration, then since $L_1$ and $L_2$ preserve cofibrations and $C$ is left proper
	it follows that if the components of the natural transformation associated to
	$X_1,X_2,Y_1$ are weak equivalences so it is the component at $Y_2$. 
	
	Any marked simplicial set $X$ with finitely many non degenerate simplices can by obtained
	from $X^{\flat}$ by a finite sequence of pushouts along the monomorphism $(\Delta^1)^{\flat}
	\to (\Delta^{1})^{\sharp}$. This implies that we can reduce to the case of minimally marked
	simplicial sets.
	
	We claim that if the natural transformation is a weak equivalence at $(\Delta^n)^{\flat}$
	for $n\geq 0$, then the result follows in general. In order to see this, we proceed by
	induction on the dimension of our simplicial sets. The case $n=0$ is clearly true. We prove
	now the case $n$, provided that the result holds for $n-1$. Since we can assume that our
	simplicial set $X$ has finitely many non degenerate simplices, we can express it as an
	iterated pushout of $\sk_{n-1}(X)$ along the cofibration $(\partial \Delta^{n})^{\flat} \to
	(\Delta^n)^{\flat}$ and the claim holds.
	
	We claim that the natural transformation is a weak equivalence for $(\Delta^n)^{\flat}$ with
	$n \geq 0$. We proceed by induction. The claim holds for $n \leq 1$ by hypothesis. Suppose
	that the claim holds for $n-1$. Then an iterated pushout along cofibrations shows that the
	component of the natural transformation at $(\Lambda^{n}_i)^{\flat}$ is a weak equivalence
	for every $0 \leq i \leq n$. Observe that in the model structure associated to
	$\overline{K}$, minimally marked inner horn inclusions are trivial cofibrations. It follows
	that the image of those maps under $L_1$ is a trivial cofibration, and by assumption the same is true of $L_2$. Therefore the
	result follows from commutativity of the diagram
	\[
	\begin{tikzcd}
	L_1((\Lambda^{n}_i)^{\flat}) \arrow[r] \arrow[d] & L_1((\Delta^n)^{\flat}) \arrow[d] \\
	L_2((\Lambda^{n}_i)^{\flat}) \arrow[r] & L_2((\Delta^n)^{\flat})
	\end{tikzcd}
	\]
	and 2-out-of-3.
\end{proof}

\begin{remark}
	Under the assumption that $L_2$ is also a left Quillen functor, preservation of trivial
	cofibrations implies that the morphisms \ref{eq:hornfillcond} are mapped to weak
	equivalences. Consequently, Lemma \ref{lem:eq} yields a criterion allowing us to quickly
	check when two left Quillen functors are equivalent. We will reapply the lemma in this
	situation when proving Theorem \ref{thm:3} below.
\end{remark}

\begin{corollary}\label{cor:chiisQE}
	Let $\CC$ be a $2$-category. Then the adjunction 
	\[
	\PPhi_{\CC}: (\msSet)_{/\Nsc(\CC)} \llra \Fun_{\msSet}(\CC^{(\op,\op)}, \msSet): \cchi_{\CC},
	\]
	 is a Quillen equivalence. 
\end{corollary}
\begin{proof}
	By Proposition \ref{prop:trfib}, $\cchi_{\CC}$ preserves trivial fibrations so that $\PPhi_{\CC}$ preserves cofibrations. As a
	left Quillen functor, the functor $\St_{\CC}$ preserves weak equivalences. By
	$2$-out-of-$3$, we deduce from Proposition \ref{prop:objectwisewe} the $\PPhi_{\CC}$
	preserves weak equivalences as well. In particular, $\PPhi_{\CC}$ preserves trivial
	cofibrations, so that $\PPhi_{\CC}$ is a left Quillen functor. The fact that $\PPhi_{\CC}$
	is in fact a left Quillen equivalence also follows immediately from Proposition
	\ref{prop:objectwisewe}, since $\St_{\CC}$ is a left Quillen equivalence by \cite[Thm
	3.8.1]{LurieGoodwillie}.
\end{proof}

It is immediate that $\cchi_{\CC}$ defines an equivalence between the $\infty$-categories associated
to the model categories $\Fun_{\msSet}(\CC^{(\op,\op)}, \msSet)$ and $(\msSet)_{/\Nsc(\CC)}$,
respectively, by means of $\infty$-categorical localization.

Since both model categories are in fact simplicial model categories, the $\infty$-categories thus
obtained can be described explicitly as the coherent nerves of the simplicial subcategories of
fibrant-cofibrant objects. In this context, it becomes desirable to lift the Quillen
equivalence to a simplicially enriched Quillen equivalence so as to obtain a description of the
equivalences induced by $\cchi_{\CC}$ and $\PPhi_{\CC}$ via coherent nerves.  Unfortunately, this is
not possible. However, we at least have the following result: 

\begin{proposition}
	The functor $\cchi_{\CC}$ can be naturally extended to a simplicially enriched functor
	$\widetilde{\cchi_{\CC}}$. In particular, we obtain the explicit description
	\[
		\N(\widetilde{\cchi_{\CC}}): \N(\Fun_{\msSet}(\CC^{(\op,\op)},\msSet)^\circ) \lra
		\N((\msSet)_{/\Nsc(\CC)}^\circ)
	\]
	of the equivalence of $\infty$-categories induced by $\cchi_{\CC}$. 
\end{proposition}
\begin{proof}
	We show that the simplicial enrichment $\widetilde{\cchi_{\CC}}$ exists.  
	Since both model categories are simplicial model categories the remaining statement follows
	from \cite[Remark A.3.1.9]{LurieHTT}, \cite[Proposition A.3.10]{LurieHTT} and Corollary
	\ref{cor:chiisQE}.
	Recall that the simplicial enrichment of the category 
	\[
		\Fun_{\msSet}(\CCop, \msSet)
	\]
	is adjoint to the tensor structure given by the formula
	\[
		(K , F) \mapsto K^{\sharp} \times F
	\]
	where the Cartesian product is taken pointwise. We then provide a map of simplicial sets
	\[
		\widetilde{\cchi_{\CC}}: \Map(F,G) \lra \Map(\cchi_{\CC}(F),\cchi_{\CC}(G)) 
	\]
	as follows: Given an $n$-simplex $\sigma: \Delta^n \to \Map(F,G)$, we apply $\cchi_{\CC}$ to the
	adjoint map
	\[
		(\Delta^n)^{\sharp} \times F \lra G
	\]
	to obtain
	\[
		\cchi_*((\Delta^n)^{\sharp}) \times \cchi_{\CC}(F) \mapsto \cchi_{\CC}(G)
	\]
	where $\cchi_*$ denotes the $2$-nerve relative to the final $2$-category $*$. We then
	precompose with the map $(\Delta^n)^{\sharp} \to \cchi_*((\Delta^n)^{\sharp})$, induced by
	pullback along the morphisms $p_I$ of posets from Equation \ref{eq:p} below to obtain a map
	\[
		(\Delta^n)^{\sharp} \times \cchi_{\CC}(F) \mapsto \cchi_{\CC}(G)
	\]
	and finally define $\widetilde{\cchi_{\CC}}(\sigma)$ to be the $n$-simplex adjoint to this latter map. This
	construction defines a map $\widetilde{\cchi_{\CC}}$ of simplicial sets and provides the desired
	$\sSet$-enrichment of $\cchi_{\CC}$.
\end{proof}

\subsection{Comparison of the $1$-categorical and $2$-categorical relative nerves}
\label{sec:1nerve}

We conclude with a comparison between the relative $1$-nerve and the relative $2$-nerve in the case
when $\CC = C$ is a $1$-category.
To relate the two nerves, consider, for an ordinal $I \subset [n]$, the map of posets \glsadd{pI}
\begin{equation}\label{eq:p}
	p_I: D^I \to I,\; S \mapsto \max(S)
\end{equation}
and denote its nerve by $\pi_I$. Given an $n$-simplex of $\chi(F)$, comprised of a
collection of compatible maps
\[
	\Delta^I \mapsto F(\sigma(\min(I))), 
\]
parametrized by all nonempty subsets $I \subset [n]$, we produce, by pulling back along the
various maps $\pi_I$, a collection of maps
\[
	\D^I \mapsto F(\sigma(\min(I))).
\]
Unravelling the definitions, using that $C$ is a $1$-category, it follows that this
collection is in fact compatible and defines an $n$-simplex of $\chi(F)$. Compatibility with
the simplicial structures yields a comparison map
\begin{equation}\label{eq:mapcomparison}
	\pi^{*}:\chi_{C}(F) \lra \cchi_{C}(F)
\end{equation}
of marked simplicial sets over $\N(C)$, natural in $F$. 

\begin{theorem}\label{thm:3}
	Let $C$ be a $1$-category. Then, for every functor $F: C^{\op} \to \Cat_{\infty}$,
	the morphism
	\[
		\pi^*: \chi_C(F) \lra \cchi_C(F),
	\]
	from \eqref{eq:mapcomparison} is an equivalence of Cartesian fibrations over
	$\N(C)$.
\end{theorem}
\begin{proof}
	Since both $\cchi_C(F)$ and $\chi_C(F)$ are Cartesian fibrations over $\N(C)$, it suffices to
	show that the morphism $\pi^*$ induces a marked equivalence of each fiber. By base change,
	we thus reduce to the case when $C$ is the final $1$-category. In this situation, the
	relative $1$-nerve is part of the Quillen equivalence 
	\[
		\phi_*: \msSet \llra \msSet: \chi_*
	\]
	where both $\phi$ and $\chi$ are the identity functors. The relative $2$-nerve is the right
	Quillen functor of the Quillen equivalence 
	\[
		\PPhi_*: \msSet \llra \msSet: \cchi_*
	\]
	from Corollary \ref{cor:chiisQE}.
	To show that the natural transformation $\pi^*: \chi_* \Rightarrow \cchi_*$ is a weak equivalence on fibrant
	objects, it suffices to show that the adjoint transformation $\pi_!: \PPhi_* \Rightarrow \phi_*$ is a weak
	equivalence. By Lemma \ref{lem:eq}, it suffices to verify that $\pi_!$ is a weak equivalence
	on the marked simplicial sets $(\Delta^0)^\flat$, $(\Delta^1)^{\flat}$, and
	$(\Delta^1)^{\sharp}$. In these cases, $\pi_!$ induces an isomorphism of marked
	simplicial sets, concluding the argument.
\end{proof}

\end{document}